\documentclass{amsart}
\usepackage{amssymb, amsmath, amsthm, graphics, comment, xspace, enumerate}
\usepackage{hyperref}
\usepackage{color}
\newtheorem{thm}{Theorem}[section]

\newtheorem{prop}[thm]{Proposition}
\theoremstyle{definition}
\newtheorem{defn}[thm]{Definition}
\newtheorem{example}[thm]{Example}
\theoremstyle{remark}
\newtheorem{rem}[thm]{Remark}
\numberwithin{equation}{section}

\begin{document}
\title[Multidimensional vector-valued $Z$-transform...]{Multidimensional vector-valued $Z$-transform and its applications}

\author{Marko Kosti\' c}
\address{Faculty of Technical Sciences,
University of Novi Sad,
Trg D. Obradovi\' ca 6, 21125 Novi Sad, Serbia}
\email{marco.s@verat.net}

{\renewcommand{\thefootnote}{} \footnote{
\\ \text{  }  \ \    {\it Key words and phrases.} Multidimensional vector-valued $Z$-transform, abstract partial difference equations, abstract Volterra difference equations with multiple variables, multidimensional discrete convolution products, abstract fractional partial difference equations.}}

\begin{abstract}
In this paper, we systematically
investigate the multidimensional $Z$-transform of functions with values in sequentially complete locally convex spaces over the field of complex numbers. We provide many structural characterizations, remarks and applications of established results to abstract Volterra difference equations depending on several variables. We also consider multidimensional discrete convolution products in vector-valued setting.
\end{abstract}
\maketitle

\section{Introduction and preliminaries}

Discrete fractional calculus, discrete fractional equations and discrete Volterra equations are rapidly growing fields of research (cf. the  monographs \cite{abbasi} by S. Abbas et al., \cite{fere} by R. A. C. Ferreira, \cite{dfc} by
C. Goodrich, A. C. Peterson, the research articles \cite{atici1}-\cite{atici3} by F. M. Atici and P. W. Eloe, and the research articles quoted in the recent monograph \cite{funkcionalne} by M. Kosti\' c). Discrete fractional calculus is incredibly important in studies of neural networks, complex dynamic systems, frequency response analysis, image processing and interval-valued systems, among many other fields of pure and applied science. Concerning the partial difference equations, we refer the reader to the monographs \cite{sscheng} by S. S. Cheng, the edited book \cite{petro} by E. N. Petropoulou, and the research monograph \cite{delaydiff} by B. Zhang, Y. Zhou.  

$Z$-Transform  (discrete Laplace transform) and its applications have been investigated by many authors (cf. the  monographs \cite{jury1}-\cite{jury} by  E. I. Jury, \cite{chip} by Y. Z. Tsypkin, \cite{grove} by A. C. Grove, Chapter 3 in the monograph \cite{4949} by J. G. Proakis, D. G. Manolakis as well as the research article \cite{bohnert} and  the list of references quoted therein). 
If $(X,\| \cdot \|)$ is a complex Banach space and a sequence $(f_{k})_{k\in {\mathbb N}_{0}}$ in $X$ satisfies $\limsup_{k\rightarrow +\infty}\| f_{k}\|^{1/k}<r<+\infty,$ then the function\index{$Z$-transform of sequences}
$$
F(z):=F\bigl\{f_{k}\bigr\}(z):=\sum_{k=0}^{\infty}\frac{f(k)}{z^{k}},\quad |z|>r
$$
is analytic and $F(\cdot)$ is called $Z$-transform of $(f_{k})_{k\in {\mathbb N}_{0}}.$ If $(B_{k})_{k\in {\mathbb N}_{0}}$ is a sequence in $L(X)$, where $L(X)$ denotes the Banach space of all bounded linear operators on $X$, and $\limsup_{k\rightarrow +\infty}\| B_{k}\|^{1/k}<r<+\infty,$ then we can also consider the $Z$-transform of $(B_{k})_{k\in {\mathbb N}_{0}},$ which is defined by
$$
\Phi(z):=\sum_{k=0}^{\infty}\frac{B(k)}{z^{k}},\quad |z|>r.
$$
Then a simple computation yields that the sequence $((B\ast_{0} f)(k))_{k\in {\mathbb N}_{0}}$ satisfies $\limsup_{k\rightarrow +\infty}\| (B\ast_{0} f)(k)\|^{1/k}<r<+\infty$, where $$(B\ast_{0} f)(k):=\sum_{j=0}^{k}B(k-j)f(j),\quad k\in {\mathbb N}_{0},$$ and that its $Z$-transform is given by $F(\cdot)\Phi(\cdot).$ 
Applying the Cauchy formula for the coefficients of Laurent series, we obtain the formula for the inverse $Z$-transform:
$$
f(k)=\frac{1}{2\pi i}\oint_{|z|=r_{0}}z^{k-1}F(z)\, dz,\quad k\in {\mathbb N}_{0},
$$
where $0<r<r_{0}<+\infty.$

There are many intriguing applications of $Z$-transform in the theory of difference equations. For example, applying the $Z$-transform and the formula
\begin{align}\label{shift1}
Z\{f_{k+j}\}(z)=z^{j}\Biggl[Z\{f_{k}\}(z)-\sum_{s=0}^{j-1}f_{s}z^{-s} \Biggr],\quad |z|>r,
\end{align} 
we can prove
that the unique
solution of the abstract higher-order difference equation
$$
\sum_{k=0}^{m}A_{m-k}u(k+j)=f_{j},\quad j\in {\mathbb N}_{0}\ \ ; \ \ u(k)=0,\quad 0\leq k\leq m-1
$$
can be represented by $u(m)=(G\ast_{0} f)(m),$ where
\begin{align}\label{giligili}
G(k)=\frac{1}{2\pi i}\oint_{|z|=r}z^{k-1}\Biggl(\sum_{s=0}^{m}z^{m-s}A_{s}\Biggr)^{-1}\, dz,
\end{align}
provided that $\limsup_{k\rightarrow +\infty}\| f_{k}\|^{1/k}<r<+\infty$ and the operator pencil $\sum_{s=0}^{m}z^{m-s}A_{s}$ is boundedly invertible for $|z|\geq r;$ see \cite[Theorem 13.2.1]{gil} and \cite{fcaa2025}.

We can also analyze the bilateral $Z$-transform of a sequence $(f_{k})_{k\in {\mathbb Z}}$ in $X$, which is given by
$$
F_{b}(z):=\sum_{k=-\infty}^{\infty}\frac{f(k)}{z^{k}},\quad |z|>r.
$$
The bilateral $Z$-transform is a linear transform which is compatible with the infinite convolution of sequences and has certain time shifting properties.  

In the  multidimensional  setting, the unilateral $Z$-transform of a sequence $(f (k))_{ k\in {\mathbb N}_{0}^{n}}$  is defined by
$$
{\displaystyle F\bigl(z_{1},z_{2},\ldots ,z_{n}\bigr):=\sum _{k_{1}=0}^{\infty }\cdots \sum _{k_{n}=0 }^{\infty }f\bigl(k_{1},k_{2},\ldots ,k_{n}\bigr)z_{1}^{-k_{1}}z_{2}^{-k_{2}}\ldots z_{n}^{-k_{n}}},
$$
while the bilateral $Z$-transform of a scalar-valued sequence $(f (k))_{ k\in {\mathbb Z}^{n}}$  is defined by
$$
{\displaystyle F\bigl(z_{1},z_{2},\ldots ,z_{n}\bigr):=\sum _{k_{1}=-\infty }^{\infty }\cdots \sum _{k_{n}=-\infty }^{\infty }f\bigl(k_{1},k_{2},\ldots ,k_{n}\bigr)z_{1}^{-k_{1}}z_{2}^{-k_{2}}\ldots z_{n}^{-k_{n}}}.
$$
The discrete multidimensional Fourier transform is a special case of the bilateral multidimensional $Z$-transform with
$
{\textstyle z_{j}=e^{iw_{j}}}$ for $1\leq j\leq n.$ 
Concerning the multidimensional $Z$-transform of scalar-valued functions and its applications to the partial difference equations, we refer the reader to the research articles \cite{gregor} by J. Gregor, \cite{tensor} by S. Yu Chang, H.-C. Wu, as well as \cite[Section 2.7]{jury} and the doctoral dissertation of P. Alper \cite{alper}. 
For example, the partial difference equation 
\begin{align}\label{sho}
f\bigl(k_{1}+1,k_{2}+1\bigr)-qf\bigl(k_{1},k_{2}+1\bigr)-pf\bigl(k_{1},k_{2}\bigr)=0;\quad u\bigl(k_{1},0\bigr)=q^{k_{1}},\ k_{1}\in {\mathbb N}_{0},
\end{align}
where $p,\ q\in (0,1)$ and $p+q=1,$ appears in the probability theory. The double  $Z$-transform of a sequence $(f (k))_{ k\in {\mathbb N}_{0}^{2}}$ satisfying \eqref{sho} has to be given by
$$
F\bigl(z_{1},z_{2}\bigr)=\frac{1}{1-p z_{2}^{-1}\bigl(z_{1}-q \bigr)^{-1}}\sum_{s=0}^{+\infty}\bigl(q/z_{1}\bigr)^{s}.
$$
Performing the inverse $Z$-transform in two dimensions, we get
$$
f\bigl(k_{1},k_{2}\bigr)=p^{k_{2}}q^{k_{1}-k_{2}}\binom{k_{1}}{k_{2}},\quad \bigl(k_{1},k_{2}\bigr)\in {\mathbb N}_{0}^{2};
$$
cf. \cite[pp. 74--76]{jury} for more details. Concerning some applications of multidimensional $Z$-transform in the theory of multidimensional digital signal processing, we refer the reader to the research monograph
\cite{dud} by D. E. Dudgeon and R. M. Mersereau; 
also, the multidimensional $Z$-transform can be useful in deriving new combinatorial
identities; see, e.g., \cite[Example 3.2.4]{gregor} and the research article \cite{zeil} by D. Zeilberger. 

The multidimensional $Z$-transform of sequences with values in Banach spaces (locally convex spaces) has not been well-explored in the existing literature. This fact has strongly influenced us to write this paper, which seems to be the first significant research study of multidimensional $Z$-transform of vector-valued sequences. We introduce and analyze here the multidimensional $Z$-transform of sequences with values in sequentially complete locally convex spaces over the field of complex numbers. We provide several illustrative examples and applications of established results to abstract partial difference equations and abstract Volterra difference equations depending on several variables. We also introduce a new discrete convolution product $a\ast_{{\rm D}}^{l,j}b$ and compute its multidimensional $Z$-transform.

The structure of paper can be briefly described as follows. After explaining the basic notation and terminology used in the paper, we briefly consider holomorphic functions of several complex variables and generalized Weyl $(a,m)$-fractional derivatives. The main structural properties of multidimensional $Z$-transform of vector-valued functions are given in Section \ref{multi-z}. The basic notion is introduced in Definition \ref{manule}; in Proposition \ref{linear} and Proposition \ref{linear1}, we consider the linearity of multidimensional vector-valued $Z$-transform and the separation of variables in the multidimensional vector-valued $Z$-transform, respectively. 
The modulation property is clarified in Proposition \ref{manulakis0} and the shifting property is clarified in Propostion \ref{manulakis}; in Theorem \ref{seka}, we establish the inversion formula for the multidimensional $Z$-transform of vector-valued functions. 

Section \ref{disc} is devoted to the study of discrete convolution products and their relations with the multidimensional $Z$-transform. The basic notion is introduced in Definition \ref{dcp}, where we define the notion of a discrete convolution product $(a \ast_{{\rm D}} b)(\cdot)$ of a sequence $a:D' \rightarrow {\mathbb C}$ and a sequence $b: D'' \rightarrow X$, where $\emptyset \neq D' \subseteq {\mathbb Z}^{n}$, $\emptyset \neq D'' \subseteq {\mathbb Z}^{n}$ and ${\rm D}=(D',D'');$ the domain of sequence $(a \ast_{{\rm D}} b)(\cdot)$ is $D:=D'+D''$. The discrete convolution product $(a \ast_{{\rm D}} b)(\cdot)$ is compatible with the multidimensional $Z$-transform of sequences, as shown in Theorem \ref{dcp}. 

Further on, in our research study \cite{mvlt}, we have introduced the convolution product $a\ast_{0}^{l,j}u$ of a function $a\in L_{loc}^{1}([0,+\infty)^{l})$ and a function $u\in L_{loc}^{1}([0,+\infty)^{n} : X)$, where
$0\leq l\leq n,$ $a_{l,n}:={n\choose l},$ $1\leq j\leq a_{l,n}$, $D_{l,j}=\{j_{1},...,j_{l}\}$ is a fixed subset of $\{1,...,n\}$, and $1\leq j_{1}<...<j_{l}\leq n.$ This definition goes as follows: we set $a_{l,j}\ast_{0}^{l,j}u:=u$, if $l=0 $ and $j=1;$ if $1\leq l\leq n$, then we set
\begin{align*}  &
\Bigl( a_{l,j}\ast_{0}^{l,j}u\Bigr)\bigl(t_{1},...,t_{n}\bigr):=\int^{t_{j_{1}}}_{0}...\int^{t_{j_{l}}}_{0}a_{l,j}\bigl( t_{j_{1}}-s_{j_{1}},..., t_{j_{l}}-s_{j_{l}}\bigr)
\\ & \times u\Bigl(t_{1},...,t_{j_{1}-1},s_{j_{1}},t_{j_{1}+1},...,t_{j_{2}-1},s_{j_{2}},t_{j_{2}+1},...,t_{j_{l}-1}, s_{j_{l}},t_{j_{l}+1},...,t_{n}\Bigr)\, ds_{j_{1}}...\, ds_{j_{l}},
\end{align*}
for any $t=(t_{1},...,t_{n}) \in [0,+\infty)^{n}.$ It is clear that this notion can be simply extended to the functions $a\in L_{loc}^{1}(\Omega')$ and a function $u\in L_{loc}^{1}(\Omega'' : X)$, where $\emptyset \neq \Omega'\subseteq [0,+\infty)^{l}$ is a Lebesgue measurable set and $\emptyset \neq \Omega''\subseteq [0,+\infty)^{n}$ is a Lebesgue measurable set. In Subsection \ref{ima}, we introduce and analyze the discrete analogue of the convolution product $a\ast_{0}^{l,j}u$ and provide an extension of the notion introduced in Definition \ref{dcp}, which can be obtained by plugging $l=n$ and $j=1$ in Definition \ref{dcplj} below. The compatibility of introduced discrete convolution product $(a \ast_{{\rm D}}^{l,j}b) (\cdot)$ with the multidimensional $Z$-transform of sequences is clarified in Theorem \ref{dcplj}. 

Some applications of multidimensional vector-valued $Z$-transform to abstract Volterra difference equations with multiple variables are presented in Section \ref{sire}, which is broken down into three separate subsections:
 
If $D$ is a finite subset of ${\mathbb Z}^{n}$ and $\emptyset \neq D' \subseteq {\mathbb Z}^{n}$, then the analysis of a linear partial difference equation
\begin{align}\label{note}
\sum_{j\in D}a_{j}u(k+j)=f(k),\quad k\in D',
\end{align}
where $f : D' \rightarrow {\mathbb C}$ and $a_{j}\in {\mathbb C}$ for all $j\in D,$ is far from being trivial and the theory of multidimensional $Z$-transform is not capable of bulding a fairly complete picture of the problem. For example, if a sequence $u : D'\rightarrow {\mathbb C}$ is a solution to \eqref{note} and the complex numbers $\lambda_{1}\in {\mathbb C} \setminus \{0\},...,\lambda_{n}\in {\mathbb C} \setminus \{0\}$ satisfy $\sum_{j\in D}a_{j}\lambda_{1}^{\beta_{1}}\cdot ... \cdot \lambda_{n}^{\beta_{n}}=0$, then the 
sequence $v : D' \rightarrow {\mathbb C}$, given by $v(k):=u(k)+\lambda_{1}^{k_{1}}\cdot ... \cdot \lambda_{n}^{k_{n}}, $ $k=(k_{1},...,k_{n})\in D'$,
is also a solution of \eqref{note}. A necessary and sufficient condition for equation \eqref{note} to have a unique solution for every sequence $f : D' \rightarrow {\mathbb C}$ is discussed in \cite[Theorem 3.2.2]{gregor}; cf. also \cite[Theorem 4.1.1a, Theorem 4.1.1]{gregor}, where J. Gregor has analyzed the existence and uniqueness of so-called recursively computable solutions to \eqref{note}. 

Further on, the characteristic polynomial of a difference equation \eqref{note}
is defined by
$
P(\lambda):=\sum_{j \in D}a_{j}\lambda^{j},\ \lambda \in {\mathbb C}^{n},
$
where $\lambda^{j}\equiv \lambda_{1}^{j_{1}}\cdot \lambda_{2}^{j_{2}}\cdot ... \cdot \lambda_{n}^{j_{n}}.$ Using the notion of amoeba of the characteristic polynomial of \eqref{note} and the notion of a multiple Laurent series, E. K. Leinartas has established
a description for the solution space of a multi-dimensional difference equation with constant
coefficients \eqref{note}, where $D$ is a finite subset of ${\mathbb N}_{0}^{n}$ and $D' ={\mathbb N}_{0}^{n}$; see \cite[Theorem 1, Theorem 2]{198}.

In Subsection \ref{subs2} we analyze the operator-valued extension of problem \eqref{note}:
\begin{align*}
\sum_{j\in D}A_{j}u(k+j)=f(k),\quad k\in D',
\end{align*}
where $X$ is a sequentially complete locally convex space over the field of complex numbers, $f : D' \rightarrow X$ and $A_{j}$ is a linear operator on $X$ for all $j\in D.$ Using the multidimensional vector-valued $Z$-transform, we present two noteworthy results concerning the existence and uniqueness of solutions to the above problem; cf. Theorem \ref{moves}, Theorem \ref{moves1}, as well as Example \ref{zan} and Example \ref{njas} below. Unfortunately, the results established in \cite{198} cannot be so simply transferred to the vector-valued setting.

Subsection \ref{subs1} investigates the following abstract multi-term Volterra difference equation on ${\mathbb Z}^{n}:$
\begin{align*}\notag & 
Bu(k)+\sum_{j_{1}=0}^{l_{1}}A_{1,j_{1}}\Bigl( a_{1,j_{1}} \ast_{(D_{1},{\mathbb Z}^{n})}u \Bigr)\bigl(k+k_{1,j_{1}}\bigr)
\\& +...+\sum_{j_{s}=0}^{l_{s}}A_{s,j_{s}}\Bigl( a_{s,j_{s}} \ast_{(D_{s},{\mathbb Z}^{n})} u\Bigr)\bigl(k+k_{s,j_{s}}\bigr)=f(k),\ k\in {\mathbb Z}^{n},
\end{align*}
where the following condition holds:
\begin{itemize}
\item[(C1)] $f: {\mathbb Z}^{n} \rightarrow X$, $s\in {\mathbb N},$ $\emptyset \neq D_{j}\subseteq {\mathbb Z}^{n}$ and $l_{j}\in {\mathbb N}_{0}$ ($1\leq j\leq s$), 
$B$ is a closed linear operator on $X$, $a_{j,k} : D_{j}\rightarrow {\mathbb C}$ and $A_{j,k}$ is a closed linear operator on $X$ ($1\leq j\leq s;$ $0\leq k\leq l_{j}$).
\end{itemize} 
The main results of Subsection \ref{subs1} are Theorem \ref{movesvolt} and Theorem \ref{iniqvolt}.

In our recent research article \cite{jan}, we have considered various classes of the abstract non-scalar Volterra difference equations. In particular, we have investigated some classes of the abstract multi-term fractional difference equations with Weyl fractional derivatives, like
\begin{align*}
\bigl( \Delta^{m} Bu\bigr)(v)&=A_{1}\Bigl(\Delta^{\alpha_{1}}_{W}u\Bigr)(v+v_{1})+...
\\& +A_{n}\Bigl(\Delta^{\alpha_{n}}_{W}u\Bigr)(v+v_{n})+\Delta^{m}(k \circ Cf)(v)+\Delta^{m}g(v),\ v\in {\mathbb Z}
\end{align*}
and
\begin{align*}
B\bigl( \Delta^{m_{n}} h\bigr)(v)&=\sum_{j=1}^{n-1}A_{j}\Bigl(\Delta^{m_{n}-m_{j}}\Delta^{\alpha_{j}}_{W}h\Bigr)(v+v_{j})
\\&+A_{n}\Bigl(\Delta^{\alpha_{n}}_{W}h\Bigr)(v+v_{n})+(k \circ Cf)(v)+g(v),\ v\in {\mathbb Z};
\end{align*}
the notion and notation will be explained a little bit later.
We have also connected the solutions of the abstract multi-term fractional differential equation
$$
A_{n}D_{t}^{\alpha_{n}}u(t)+A_{n-1}D_{t}^{\alpha_{n-1}}u(t)+...+A_{1}D_{t}^{\alpha_{1}}u(t)=0,\quad t>0,
$$
where $0\leq \alpha_{1}<\alpha_{2}<...<\alpha_{n}$ and $D_{t}^{\alpha}u(t)$ denotes the Riemann-Liouville fractional derivative of function $u(t)$ of order $\alpha>0,$ with the solutions of the abstract multi-term fractional difference equation
\begin{align*}
A_{n}\Bigl[\Delta^{\alpha_{n}}_{W}u\Bigr](v)&+A_{n-1}\Bigl[\Delta^{\alpha_{n-1}}_{W}u\Bigr]\bigl(v+m_{n}-m_{n-1}\bigr)
\\&+...+A_{1}\Bigl[\Delta^{\alpha_{1}}_{W}u\Bigr]\bigl(v+m_{n}-m_{1}\bigr)=-g(v),\quad v\in {\mathbb Z}.
\end{align*}

The results established in \cite{jan} and the previous research studies of abstract fractional difference equations with Weyl fractional derivaties are given by using the existence of a corresponding solution operator family of operators defined on the set of non-negative integers (see, e.g., \cite{abadias}-\cite{abadias1}, \cite[Theorem 4.1, Theorem 4.3, Theorem 4.5]{jan} and the references quoted in \cite{funkcionalne}) or by applying the Poisson transform to the solutions of  a corresponding abstract  fractional differential equation (for the first steps made in this direction, we refer the reader to the paper \cite{lizama-pois} by C. Lizama). We would like to emphasize that the study carried out in this research article is probably the first research study which considers some applications of vector-valued $Z$-transform to abstract fractional difference equations with Weyl fractional derivaties; for the sake of better readibility, we have decided to restrict our investigation of this topic to the one-dimensional setting:

In Subsection \ref{ljeb}, we investigate the existence and uniqueness of solutions to the following abstract multi-term fractional difference equation with generalized Weyl $(a,m)$-fractional derivatives:{\small
\begin{align*}
A_{s}\Bigl(\Delta_{W,a_{s},m_{s}}u\Bigr)\bigl(k+k_{s}\bigr)+...+A_{1}\Bigl(\Delta_{W,a_{1},m_{1}}u\Bigr)\bigl(k+k_{1}\bigr)+A_{0}u\bigl(k+k_{0}\bigr)=Cf(k),\ k\in {\mathbb Z},
\end{align*}}where $s\in {\mathbb N}$, $A_{1},...,A_{s}$ are closed linear operators on a sequentially complete locally convex space over the field of complex numbers $X$, $k_{0},k_{1},...,k_{s}\in {\mathbb Z},$ $ f: {\mathbb Z} \rightarrow X$ is a given sequence and $\Delta_{W,a,m}u$ is the generalized Weyl $(a,m)$-fractional derivative of sequence $u(\cdot).$ The main results of Subsection
\ref{ljeb} are Theorem \ref{movesvoltw} and Theorem \ref{iniqvoltw}; cf. also Example \ref{parkw}.
\vspace{0.1cm} 

\noindent {\bf Notation and preliminaries.} If $Y$ is a Hausdorff sequentially complete locally convex space\index{sequentially complete locally convex space!Hausdorff} over the field of complex numbers, then we say that $Y$ is an SCLCS.
If $X$ and $Y$ are SCLCSs,
then the abbreviation $\circledast$ ($\circledast_{Y}$) stands for the fundamental system of seminorms\index{system of seminorms} which defines the topology of $X$ ($Y$), $L(X,Y)$ denotes the space consisting of all continuous linear mappings\index{continuous linear mapping} from $X$ into $Y$ and $L(X)\equiv L(X,X)$; if $\mathcal B$ is the family of bounded subsets\index{bounded subset} of $X$, then the space $L(X,Y)$ carries the Hausdorff locally convex topology induced
by the calibration $(p_{B})_{B\in {\mathcal B}}$ of seminorms on $L(X,Y)$, where $p_B(T):=\sup_{x\in B}p(Tx)$, $p\in\circledast$, $B\in\mathcal B$, $T\in L(X,Y)$.
It is well known that the space $L(X,Y)$ is sequentially complete provided that $X$ is barreled. By $X^{\ast}$ we denote the space $L(X,{\mathbb C}),$ equipped with the strong topology; ${\rm I}$ satnds for the identity operator on $X$ (cf. \cite{jarc}, \cite{meise} and \cite{trev} for more details about topological vector spaces and locally convex spaces). For further information concerning the integration in SCLCSs, we refer the reader to \cite{FKP}. 

The Gamma function\index{function!Gamma} will be denoted by $\Gamma(\cdot)$ and the principal branch will be always used to take the powers; define $g_{\zeta}(t):=t^{\zeta-1}/\Gamma(\zeta)$ and
$0^{\zeta}:=0$ ($\zeta>0,$ $t>0$).  
Given the numbers $s\in\mathbb R$ and $m\in {\mathbb N},$ we set $\lceil s\rceil:=\inf\{l\in\mathbb Z:s\leq l\}$, ${\mathbb N}_{m}:=\{1,...,m\} $ and ${\mathbb N}_{m}^{0}:=\{0,1,...,m\} .$ Unless stated otherwise, we will always assume henceforth that $X$ is an SCLCS and $n\in {\mathbb N}.$ \vspace{0.1cm}

\noindent {\bf Holomorphic functions of several complex variables.} There are many research monographs concerning holomorphic functions of several complex variables; for more details about vector-valued holomorphic functions depending on several complex variables, we refer the reader to the list of references quoted in the recent research article \cite{kruse} by K. Kruse. Recall, a function $f\colon\Omega\to X$, where $\Omega$ is an open subset of ${\mathbb C}^{n}$, is said to be holomorphic if for each $\lambda=(\lambda_{1},...,\lambda_{n}) \in \Omega$ there exists $\epsilon>0$ such that $L(\lambda,\epsilon):=\{(z_{1},..,z_{n})\in {\mathbb C}^{n} : \max\{ |z_{i}-\lambda_{i}| : 1\leq i\leq n\}<\epsilon \}\subseteq \Omega$  
and
$$
f(z)=\sum_{k_{1}=0}^{+\infty}...\sum_{k_{n}=0}^{\infty}a_{k_{1},...,k_{n}}\bigl( z_{1}-\lambda_{1} \bigr)^{k_{1}}\cdot .... \cdot \bigl( z_{n}-\lambda_{n} \bigr)^{k_{n}},\quad z\in L(\lambda,\epsilon),
$$
for some elements $a_{k_{1},...,k_{n}}\in X$ satisfying that for each seminorm $p\in \circledast$ and $r\in [0,\epsilon)$ we have
$$
\sum_{(k_{1},...,k_{n})\in {\mathbb N}_{0}^{n}}p\bigl(a_{k_{1},...,k_{n}} \bigr) r^{k_{1}+...+k_{n}}<+\infty .
$$ 
We know that the mapping $\lambda\mapsto f(\lambda)$, $\lambda\in\Omega$ is holomorphic if and only if it is weakly holomorphic, i.e., the mapping $\lambda\mapsto\langle x^*,f(\lambda)\rangle$, $\lambda\in\Omega$ is holomorphic for every $x^*\in X^*$.
If the mapping $f\colon\Omega\to X$ is holomorphic, then the Cauchy integral formula in polydiscs holds (see, e.g., \cite[Theorem 5.1, Theorem 5.7]{kruse}), the mapping $\lambda\mapsto f(\lambda)$, $\lambda\in\Omega$ is infinitely differentiable and we have
$$
a_{k_{1},...,k_{n}}=\frac{\partial^{(k_{1},...,k_{n})}f(\lambda)}{k_{1}!\cdot ... \cdot k_{n}!},\quad (k_{1},...,k_{n})\in {\mathbb N}_{0}^{n}.
$$  
Moreover, the Osgood lemma, the Hartogs theorem and the Weierstrass theorem hold for the vector-valued holomorphic functions of several variables.\vspace{0.1cm}

\noindent {\bf Generalized Weyl $(a,m)$-fractional derivatives.} If $\alpha >0,$ then the Ces\`aro sequence $
(c^{\alpha}(k))_{k\in {\mathbb N}_{0}}$ is defined by
$$
c^{\alpha}(k):=\frac{\Gamma(k+\alpha)}{\Gamma(\alpha)k!},\quad k\in {\mathbb N}_{0}.
$$
We know that, for every $\alpha>0$ and $\beta>0$, we have $| c^{\alpha}(k)-g_{\alpha}(k)|=O(g_{\alpha}(k)|1/k|)$, $k\in {\mathbb N};$ in particular, $ c^{\alpha}(v)\sim g_{\alpha}(k),$ $k\rightarrow +\infty.$ We define $c^{0}(\cdot)$ by $c^{0}(0):=1$ and $c^{0}(v):=0,$ $v\in {\mathbb N},$ then we have
$
c^{\alpha}\ast_{0}c^{\beta}\equiv c^{\alpha+\beta},\ \alpha,\ \beta\geq 0.
$

Recall, if $(u_{k})$ is a one-dimensional sequence in $X$, then the Euler forward difference operator $\Delta$ is defined by $\Delta u_{k}:=u_{k+1}-u_{k}.$ The operator $\Delta^{m}$ is defined inductively, or by the formula
\begin{align*}
\Delta^{m}u_{k}:=\sum_{j=0}^{m}(-1)^{m-j}\binom{m}{j}u_{k+j}.
\end{align*}

Suppose now that $a: {\mathbb N}_{0}\rightarrow {\mathbb C}$ and $f: {\mathbb Z}\rightarrow X$ are given sequences. If the series $\sum_{s=0}^{+\infty}a(s)f(v-s)$ is absolutely convergent for all $k\in {\mathbb Z}$, then we define
\begin{align*} 
\Bigl(\Delta_{W,a}f\Bigr)(k):=\sum_{s=-\infty}^{k}a(k-s)f(s)=\sum_{s=0}^{+\infty}a(s)f(k-s),\quad k\in {\mathbb Z}.
\end{align*}
Assume that the sequence $\Delta_{W,a}f: {\mathbb Z}\rightarrow X$ is well-defined and $m\in {\mathbb N}.$ Then we put
$$
\Bigl(\Delta_{W,a,m}f\Bigr)(k):=\Bigl(\Delta^{m}\Delta_{W,a}f\Bigr)(k),\quad k\in {\mathbb Z}.
$$
It is worth noting that, if $m=\lceil \alpha \rceil$ and $a\equiv c^{m-\alpha}$ for some $\alpha>0,$ then the operator $\Delta_{a,m}$ reduces to the Weyl fractional derivative $D_{W}^{\alpha}f$ of sequence $f(\cdot)$ of order $\alpha;$ cf. \cite[Definition 2.3]{abadias}. Because of that,  we will call the sequence $\Delta_{W,a,m}f$ the generalized Weyl $(a,m)$-fractional derivative of sequence $f(\cdot).$
 
For the sake of brevity, we will not conisder here the applications of multidimenisonal vector-valued $Z$-transform to the abstract fractional partial difference equations; for more deatils about multidimensional discrete fractional calculus, we refer the reader to \cite{funkcionalne}, \cite{aksi} and the references quoted therein.

\section{Multidimensional $Z$-transform of vector-valued functions}\label{multi-z}

We open this section by introducing the following notion:

\begin{defn}\label{manule}
Suppose that $\emptyset \neq D\subseteq {\mathbb Z}^{n},$ $f: D \rightarrow X$ and $\emptyset \neq \Omega \subseteq {\mathbb C}^{n}$. If, for every seminorm $p\in \circledast$ and for every $z=(z_{1},z_{2},\ldots ,z_{n})\in \Omega ,$ we have
\begin{align}\label{est12}
\sum _{k\in D }p\Bigl(f\bigl(k\bigr)\Bigr)\bigl|z_{1}\bigr|^{-k_{1}}\bigl|z_{2}\bigr|^{-k_{2}}\ldots \bigl| z_{n}\bigr|^{-k_{n}}<+\infty ,
\end{align} 
then the multidimensional $Z$-transform of sequence $(f(k))_{k\in D}$ in the region $\Omega$
is given by
\begin{align}\label{est121}
F_{f}\bigl(z_{1},z_{2},\ldots ,z_{n}\bigr):=\sum _{k\in D }f\bigl(k\bigr)z_{1}^{-k_{1}}z_{2}^{-k_{2}}\ldots z_{n}^{-k_{n}},  
\end{align}
for any $z=(z_{1},z_{2},\ldots ,z_{n})\in \Omega ;$ here, $k=(k_{1},k_{2},\ldots ,k_{n})$ for $k\in D.$
\end{defn}

It is worth noting that the sum in \eqref{est121} exists in $X$ since we have assumed that $X$ is sequentially complete and \eqref{est12} holds. For example, if for every seminorm $p\in \circledast$ we have $\sum _{k\in D }p (f (k ) )<+\infty,$ then 
$F_{f}(z_{1},z_{2},\ldots ,z_{n})$ exists in the region $|z_{1}|\geq 1,...,|z_{n}|\geq 1.$
Further on, 
\eqref{est12} holds if and only if \eqref{est12} holds with the numbers $z_{1},...,z_{n}$ replaced therein with the numbers $z_{1}\exp(i\varphi_{1}),...,z_{n}\exp(i\varphi_{n}),$ where $ \varphi_{i}\in {\mathbb R}$ ($1\leq i\leq n$). If the set $\Omega$ is open, then the mapping $F_{f} : \Omega \rightarrow X$ is holomorphic since it is weakly holomorphic (cf. also \cite[Theorem 2.1.2]{gregor}); furthermore, in this case, we have
\begin{align*}&
\partial^{v}F_{f}\bigl(z_{1},z_{2},\ldots ,z_{n}\bigr)=\sum _{k_{1}\in D_{1} }\cdot ... \cdot \sum _{k_{n}\in D_{n} }\Biggl[\prod_{i=1}^{n}\bigl(-k_{i}\bigr) \cdot \bigl(-k_{i}-1\bigr) \cdot ... \cdot \bigl(-k_{i}-v_{i}+1\bigr)\Biggr] 
\\& \times f\bigl(k_{1},k_{2},\ldots ,k_{n}\bigr)z_{1}^{-k_{1}-v_{1}}z_{2}^{-k_{2}-v_{2}}\ldots z_{n}^{-k_{n}-v_{n}},
\end{align*}
for any $z=(z_{1},z_{2},\ldots ,z_{n})\in \Omega$ and $v=(v_{1},v_{2},\ldots ,v_{n})\in {\mathbb N}_{0}^{n}.$

As in the scalar-valued setting, if the multidimensional $Z$-transform of a sequence $(f(k))_{k\in D}$ exists at a point $(z_{1}^{0},...,z_{n}^{0})\in {\mathbb C}^{n},$ then the multidimensional $Z$-transform of a sequence $(f(k))_{k\in D}$ exists at any point $(z_{1},...,z_{n})\in {\mathbb C}^{n}$ such that $|z_{1}|\geq |z_{1}^{0}|,$ ..., $|z_{n}|\geq |z_{n}^{0}|.$ But, the region of absolute convergence of multidimensional $Z$-transform can be complex, as the following illustrative examples show:

\begin{example}\label{wasp}
\begin{itemize}
\item[(i)] If $n=2$, $a>0$, $f(k_{1},k_{2}):=a^{k_{1}}$ for $k_{1}=k_{2}\in {\mathbb N}_{0}$ and $f(k_{1},k_{2}):=0$ for $k_{1},\ k_{2}\in {\mathbb N}_{0}$ with $k_{1}\neq k_{2}$, then 
\begin{align*}
F_{f}(z)=\frac{1}{1-a^{-1}z_{1}^{-1}z_{2}^{-1}},\  \mbox{ provided }\ \bigl(z_{1},z_{2}\bigr)\in {\mathbb C}^{2}\ \mbox{ and }\ \bigl| z_{1}z_{2}\bigr|>a.
\end{align*}
\item[(ii)] If $D:=\{(k_{1},k_{2}) \in {\mathbb Z}^{2}: k_{1}\geq 0,\ k_{2}\geq -k_{1}\}$ and $f(k_{1},k_{2}):=1$ for all $(k_{1},k_{2})\in D,$ then
\begin{align*}
F_{f}(z)&=\frac{z_{2}^{-1}}{(1-z_{2}^{-1})(z_{2}^{-1}-z_{1}^{-1})},\\&  \mbox{ provided }\ \bigl(z_{1},z_{2}\bigr)\in {\mathbb C}^{2},\ \bigl| z_{1}\bigr|>1,\ \bigl| z_{2}\bigr|>1\ \mbox{ and }\ \bigl| z_{1}\bigr|>\bigl| z_{2}\bigr|.
\end{align*}
\end{itemize}
\end{example}

For further information concerning the regions of absolute convergence of multidimensional $Z$-transform and the notion of Reinhardt domains, we refer the reader to \cite[pp. 175--180]{dud} and \cite[Subsection 2.2]{gregor}. 

The linearity of multidimensional $Z$-transform can be simply shown: 

\begin{prop}\label{linear}
Suppose that $\alpha,\ \beta \in {\mathbb C}$, $f: D \rightarrow X$, $g: D \rightarrow X$ and for each seminorm $p\in \circledast$ there exists a real constant $M_{p}>0$ such that the estimate \eqref{est12} holds for any $z=(z_{1},z_{2},\ldots ,z_{n})\in \Omega ,$ and both sequences $f(\cdot),$ $g(\cdot)$. Then we have 
$$
F_{\alpha f+ \beta g} (z  )=\alpha F_{f} (z  )+\beta F_{g} (z ),\quad z\in \Omega.
$$
\end{prop}

The principle of inclusion and exclusion for $m$ sets, where $m\in {\mathbb N} \setminus \{1\},$ can be simply formulated in terms of multidimensional vector-valued $Z$-transform; cf. also \cite[Theorem 2.1.4]{gregor} for case $m=2,$ where the established property was called the additivity of multidimensional $Z$-transform.

Suppose now that $r_{i}>0$, $D_{i}={\mathbb N}_{0}$ or $D_{i}=-{\mathbb N}_{0}$ ($1\leq i\leq n$), $D=D_{1}\times ... \times D_{n}$, for each seminorm $p\in \circledast$ there exists a real constant $M_{p}>0$ such that 
\begin{align}\label{est1}
p\Bigl(f\bigl(k_{1},...,k_{n}\bigr)\Bigr)\leq M_{p}\Bigl[r_{1}^{k_{1}}\cdot ... \cdot r_{n}^{k_{n}}\Bigr],\quad k=\bigl(k_{1},...,k_{n}\bigr)\in D, 
\end{align}
and $\Omega :=\Omega_{1} \times ... \times \Omega_{n}$ , where $\Omega_{i}:=\{z_{i}\in {\mathbb C} : |z_{i}|>r_{i}\},$ if $D_{i}={\mathbb N}_{0},$ and $\Omega_{i}:=\{z_{i}\in {\mathbb C} : |z_{i}|<r_{i}\},$ if $D_{i}={\mathbb N}_{0}$.
Then the integral in \eqref{est121} is convergent since it is absolutely convergent and $X$ is sequentially complete; in actual fact, if $p\in \circledast$ and $z=(z_{1},z_{2},\ldots ,z_{n})\in \Omega $, then we have{\small
\begin{align*} 
p\Bigl( F_{f}\bigl(z_{1},z_{2}, & \ldots ,z_{n}\bigr)\Bigr) \leq \sum _{k_{1}\in D_{1} }\cdot ... \cdot \sum _{k_{n}\in D_{n} }p\Bigl(f\bigl(k_{1},k_{2},\ldots ,k_{n}\bigr)\Bigr)\bigl|z_{1}\bigr|^{-k_{1}}\bigl|z_{2}\bigr|^{-k_{2}}\ldots \bigl|z_{n}\bigr|^{-k_{n}}
\\& \leq M_{p}\sum _{k_{1}\in D_{1} }\cdot ... \cdot \sum _{k_{n}\in D_{n} }r_{1}^{k_{1}}\cdot ... \cdot r_{n}^{k_{n}}\bigl|z_{1}\bigr|^{-k_{1}}\bigl|z_{2}\bigr|^{-k_{2}}\ldots \bigl|z_{n}\bigr|^{-k_{n}}
\\& = M_{p}\sum _{k_{1}\in D_{1} }\bigl( r_{1}/|z_{1}|\bigr)^{k_{1}} \cdot ... \cdot \sum _{k_{n}\in D_{n} }\bigl( r_{n}/|z_{n}|\bigr)^{k_{1}}<+\infty.
\end{align*}}

If there exists $i\in {\mathbb N}_{n}$ such that $D_{i}={\mathbb Z}$, and $D=D_{1}\times ... \times D_{n}$, then the estimate \eqref{est1} is far from being sufficient for ensuring the convergence of sum in \eqref{est121} for any point $z\in {\mathbb C}^{n}\setminus \{0\}.$ But, in this case, the bilateral multidimensional $Z$-transform $F_{f}(z)$ can be defined for every point $z\in {\mathbb C}^{n}\setminus \{0\}$; for example, if $D={\mathbb Z}^{n}  $ and $f(k_{1},...,k_{n})=\exp(-k_{1}^{2}-...-k_{n}^{2}),$ $k=(k_{1},...,k_{n})\in {\mathbb Z}^{n},$ then $F_{f}(z)$ is defined for every point $z\in {\mathbb C}^{n}\setminus \{0\}.$ 

We can simply prove the following result:

\begin{prop}\label{linear1}
\begin{itemize}
\item[(i)] Suppose that $\emptyset \neq \Omega \subseteq {\mathbb C}^{n}$, $f_{i}: D_{i} \rightarrow {\mathbb C}$ ($1\leq i\leq n-1$), $f_{n}: D_{n} \rightarrow X$, $D=D_{1}\times ... \times D_{n}$ and $f(k_{1},...,k_{n}):=f_{1}(k_{1}) \cdot ... \cdot f_{n-1}(k_{n-1})\cdot f_{n}(k_{n}),$ $k=(k_{1},...,k_{n})\in D$. If, for every $z=(z_{1},z_{2},\ldots ,z_{n})\in \Omega $, we have
$\sum_{k_{i}\in D_{i}}| f_{i}(z_{i})| \cdot |  z_{i}|^{-k_{i}}<+\infty$  ($1\leq i\leq n-1$), and
for each seminorm $p\in \circledast$ 
we have $\sum_{k_{n}\in D_{n}}p(f_{n}(z_{n}))|z_{n}|^{-k_{n}}<+\infty$, then 
\begin{align}\label{estb1}
F_{f}\bigl(z_{1},....,z_{n}\bigr)=\prod_{i=1}^{n}F_{f_{i}}\bigl(z_{i}\bigr) ,\quad z=\bigl(z_{1},....,z_{n}\bigr)\in \Omega.
\end{align}
\item[(ii)]
Suppose that $r_{i}>0$ and $D_{i}={\mathbb N}_{0}$ or $D_{i}=-{\mathbb N}_{0}$ ($1\leq i\leq n$), $f_{i}: D_{i} \rightarrow {\mathbb C}$ ($1\leq i\leq n-1$), $f_{n}: D_{n} \rightarrow X$, $D=D_{1}\times ... \times D_{n}$ and $f(k_{1},...,k_{n}):=f_{1}(k_{1}) \cdot ... \cdot f_{n-1}(k_{n-1})\cdot f_{n}(k_{n}),$ $k=(k_{1},...,k_{n})\in D$. If there exist real numbers $m_{i}>0$ such that $|f_{i}(k_{i})|\leq m_{i}r_{i}^{k_{i}}$ ($1\leq i\leq n-1$), and for each seminorm $p\in \circledast$ there exists a real constant $M_{p}>0$ such that $p(f(k_{n}))\leq M_{p}r_{n}^{k_{n}},$ $k_{n}\in D_{n}$, then \eqref{estb1} holds with $\Omega :=\Omega_{1} \times ... \times \Omega_{n}$, where $\Omega_{i}:=\{z_{i}\in {\mathbb C} : |z_{i}|>r_{i}\},$ if $D_{i}={\mathbb N}_{0},$ and $\Omega_{i}:=\{z_{i}\in {\mathbb C} : |z_{i}|<r_{i}\},$ if $D_{i}={\mathbb N}_{0}$.
\end{itemize}
\end{prop}
 
The modulation property of multidimensional $Z$-transform reads as follows; the proof is trivial and therefore omitted:
 
\begin{prop}\label{manulakis0}
Suppose that $f: D \rightarrow X$, $w: D \rightarrow X$, $\emptyset \neq \Omega \subseteq {\mathbb C}^{n}$, $a_{i}\in {\mathbb C} \setminus \{0\}$ ($1\leq i\leq n$) and $f(k_{1},...,k_{n})=a_{1}^{k_{1}}\cdot ... \cdot a_{n}^{k_{n}}\cdot w(k_{1},...,k_{n})$ for all $k=(k_{1},...,k_{n})\in D.$ If the multidimensional $Z$-transform of sequence $(w(k))_{k\in D}$ is well-defined in the region $\Omega_{w}:=\{ (z_{1}/a_{1},...,z_{n}/a_{n}) : (z_{1} ,...,z_{n} )\in \Omega \}$, then the multidimensional $Z$-transform of sequence $(f(k))_{k\in D}$ is well-defined in the region $\Omega$, and we have
\begin{align*}
F_{f}\bigl(z_{1},z_{2},\ldots ,z_{n}\bigr)= F_{w}\bigl(z_{1}/a_{1},z_{2}/a_{2},\ldots ,z_{n}/a_{n}\bigr),\quad z=\bigl(z_{1},z_{2},\ldots ,z_{n}\bigr)\in \Omega.
\end{align*}
\end{prop}

The multidimensional $Z$-transform has some expected shifting properties. In order to properly formulate the next result, which extends the formula \eqref{shift1}, the formula \cite[(2.173), p. 74]{jury} for case $n=2$ and $D_{1}=D_{2}={\mathbb N}_{0}$, and the formula \cite[(4.33b), p. 180]{dud} for case $n=2$ and $D_{1}=D_{2}={\mathbb Z}$, we set 
$z^{a}:=z_{1}^{a_{1}}\cdot z_{2} ^{a_{2}}\ldots z_{n}^{a_{n}}$ for any $z=(z_{1},...,z_{n})\in {\mathbb C}^{n}$ and $a=(a_{1},...,a_{n}) \in {\mathbb Z}^{n}$ such that the above term is well-defined ($0^{0}:=1,$ $0^{k}=0,$ $k\in {\mathbb N}$ and $0^{-k}$ is not well-defined for $k\in {\mathbb N}$):

\begin{prop}\label{manulakis}
Suppose that $f: D \rightarrow X$, $\emptyset \neq \Omega \subseteq {\mathbb C}^{n}$, $a\in {\mathbb Z}^{n}$, $a+D\subseteq D$ and, for every seminorm $p\in \circledast$ and for every $z=(z_{1},z_{2},\ldots ,z_{n})\in \Omega ,$ we have \eqref{est12}. Set $g(k_{1},...,k_{n}):=f(a_{1}+k_{1},...,a_{n}+k_{n}),$ $(k_{1},...,k_{n})\in D.$
Then the multidimensional $Z$-transform of sequence $(g(k))_{k\in D}$ in the region $\Omega$
is given by
\begin{align}\label{est121shift}
F_{g}\bigl(z_{1},z_{2},\ldots ,z_{n}\bigr)=z_{1}^{a_{1}}\cdot z_{2} ^{a_{2}}\ldots z_{n}^{a_{n}}\Biggl[ F_{f}\bigl(z_{1},z_{2},\ldots ,z_{n}\bigr)-\sum_{z\in D\setminus (a+D)}f(k)z^{-k}\Biggr],
\end{align}
for any $z=(z_{1},z_{2},\ldots ,z_{n})\in \Omega.$
\end{prop}

\begin{proof}
Is is clear that, for every seminorm $p\in \circledast$ and for every $z=(z_{1},z_{2},\ldots ,z_{n})\in \Omega ,$ we have:
\begin{align*}
& \sum_{k\in D}p(f(k+a))  \cdot \bigl|z_{1}\bigr|^{-k_{1}}\bigl|z_{2}\bigr|^{-k_{2}}\ldots \bigl| z_{n}\bigr|^{-k_{n}}
\\& = \bigl|z_{1}\bigr|^{a_{1}}\bigl|z_{2}\bigr|^{a_{2}}\ldots \bigl| z_{n}\bigr|^{a_{n}} \sum_{k\in D}p(f(k+a)) \cdot \bigl|z_{1}\bigr|^{-(k_{1}+a_{1})}\bigl|z_{2}\bigr|^{-(k_{2}+a_{2})}\ldots \bigl| z_{n}\bigr|^{-(k_{n}+a_{n})}
\\& =\bigl|z_{1}\bigr|^{a_{1}}\bigl|z_{2}\bigr|^{a_{2}}\ldots \bigl| z_{n}\bigr|^{a_{n}} \sum_{w\in a+D}p(f(w))\bigl|z_{1}\bigr|^{-w_{1}}\bigl|z_{2}\bigr|^{-w_{2}}\ldots \bigl| z_{n}\bigr|^{-w_{n}}
\\& \leq \bigl|z_{1}\bigr|^{a_{1}}\bigl|z_{2}\bigr|^{a_{2}}\ldots \bigl| z_{n}\bigr|^{a_{n}} \sum_{w\in D}p(f(w))\bigl|z_{1}\bigr|^{-w_{1}}\bigl|z_{2}\bigr|^{-w_{2}}\ldots \bigl| z_{n}\bigr|^{-w_{n}}<+\infty.
\end{align*}
Then the final conclusion simply follows from the next calculus:{\small
\begin{align*}
& F_{g}\bigl(z_{1},z_{2},\ldots ,z_{n}\bigr)=\sum _{k_{1}\in D_{1} }\cdot ... \cdot \sum _{k_{n}\in D_{n} }g\bigl(k_{1},k_{2},\ldots ,k_{n}\bigr)z_{1}^{-k_{1}}z_{2}^{-k_{2}}\ldots z_{n}^{-k_{n}}
\\& =\sum _{k_{1}\in D_{1} }\cdot ... \cdot \sum _{k_{n}\in D_{n} }f\bigl(a_{1}+k_{1},a_{2}+k_{2},\ldots ,a_{n}+k_{n}\bigr)z_{1}^{-k_{1}}z_{2}^{-k_{2}}\ldots z_{n}^{-k_{n}}
\\& =  z_{1}^{a_{1}}\cdot z_{2} ^{a_{2}}\ldots z_{n}^{a_{n}}\sum _{k_{1}\in D_{1} }\cdot ... \cdot \sum _{k_{n}\in D_{n} }f\bigl(a_{1}+k_{1},a_{2}+k_{2},\ldots ,a_{n}+k_{n}\bigr)
\\& \times z_{1}^{-(a_{1}+k_{1})}z_{2}^{-(a_{2}+k_{2})}\ldots z_{n}^{-(a_{n}+k_{n})}
\\& =z^{a}\sum_{k\in D}f(a+k)z^{-(k+a)}=z^{a}\sum_{w\in a+D}f(w)z^{-w}
\\& =z^{a}\Biggl[ F_{f}(z)-\sum_{w\in D\setminus (a+D)}f(w)z^{-w}\Biggr]
\\& =z_{1}^{a_{1}}\cdot z_{2} ^{a_{2}}\ldots z_{n}^{a_{n}}\Biggl[ F_{f}\bigl(z_{1},z_{2},\ldots ,z_{n}\bigr)-\sum_{z\in D\setminus (a+D)}f(k)z^{-k}\Biggr],
\end{align*}}
wich holds for any $z=(z_{1},z_{2},\ldots ,z_{n})\in \Omega.$
\end{proof}

Now we will clarify the following inversion type theorem for the multidimensional $Z$-transform of vector-valued functions:

\begin{thm}\label{seka}
Suppose that $\emptyset \neq \Omega \subseteq {\mathbb C}^{n}$, $r_{i}>0$ for $1\leq i\leq n$ and $\Omega$ contains an open neighborhood of the set $\{z_{1}\in {\mathbb C} : |z_{1}|=r_{1}\} \times ... \times \{z_{n}\in {\mathbb C} : |z_{n}|=r_{n}\}$ in ${\mathbb C}^{n}.$ If the multidimensional $Z$-transform of sequence $(f(k))_{k\in D}$ is well-defined in the region $\Omega$, then we have
\begin{align}\label{est121inv}
f(k)=\frac{1}{(2\pi i)^{n}}\oint_{|z_{1}|=r_{1}}\cdot ... \cdot \oint_{|z_{n}|=r_{n}}z_{1}^{k_{1}-1}\cdot ... \cdot z_{n}^{k_{n}-1}
F_{f}\bigl(z_{1},z_{2},\ldots ,z_{n}\bigr)\, dz_{1}\, ...\, dz_{n},
\end{align}
for any $k=(k_{1},...,k_{n})\in D.$
\end{thm}

\begin{proof}
Using the appropriate functionals, we may assume without loss of generality that the sequence $f(\cdot)$ is scalar-valued. Then we have
\begin{align}& \notag
\frac{1}{(2\pi i)^{n}}\oint_{|z_{1}|=r_{1}}\cdot ... \cdot \oint_{|z_{n}|=r_{n}}z_{1}^{k_{1}-1}\cdot ... \cdot z_{n}^{k_{n}-1}
F_{f}\bigl(z_{1},z_{2},\ldots ,z_{n}\bigr)\, dz_{1}\, ...\, dz_{n}
\\\notag& =\frac{1}{(2\pi i)^{n}}\oint_{|z_{1}|=r_{1}}\cdot ... \cdot \oint_{|z_{n}|=r_{n}}z_{1}^{k_{1}-1}\cdot ... \cdot z_{n}^{k_{n}-1}\sum_{l\in D}f(l)z_{1}^{-l_{1}}\cdot ... \cdot z_{n}^{l_{n}}\, dz_{1}\, ...\, dz_{n}
\\\notag& =\frac{1}{(2\pi i)^{n}}\oint_{|z_{1}|=r_{1}}\cdot ... \cdot \oint_{|z_{n}|=r_{n}}\sum_{l\in D}f(l)z_{1}^{k_{1}-1-l_{1}}\cdot ... \cdot z_{n}^{k_{n}-1-l_{n}}\, dz_{1}\, ...\, dz_{n}
\\\label{jazz}& =\frac{1}{(2\pi i)^{n}}\sum_{l\in D}f(l)\oint_{|z_{1}|=r_{1}}\cdot ... \cdot \oint_{|z_{n}|=r_{n}}z_{1}^{k_{1}-1-l_{1}}\cdot ... \cdot z_{n}^{k_{n}-1-l_{n}}\, dz_{1}\, ...\, dz_{n}
\\\notag & =f(k),\quad k=(k_{1},...,k_{n})\in D,
\end{align}
where the last equality follows from the Fubini theorem and the residue theorem (the equality \eqref{jazz} can be simply justified since the sum in \eqref{jazz} is finite and consists of only one term).
\end{proof}

For example, if $a$ and $b$ are complex numbers such that $|a|+|b|<1,$ then the inverse $Z$-transform of function
$$
F_{f}\bigl(z_{1},z_{2}\bigr)=\frac{z_{1}z_{2}}{z_{1}z_{2}-az_{2}-bz_{1}},\quad \bigl| z_{1}\bigr|>1-\epsilon,\ \bigl| z_{2}\bigr|>1-\epsilon,
$$
 where the number $\epsilon>0$ is sufficiently small, is given by
$$
f\bigl(k_{1},k_{2}\bigr)=a^{k_{1}}b^{k_{2}}\frac{(k_{1}+k_{2})!}{k_{1}! \cdot k_{2}!}u\bigl(k_{1}\bigr)u\bigl(k_{2}\bigr),\quad k=\bigl(k_{1},k_{2}\bigr)\in {\mathbb N}_{0}^{2},
$$
where $u(l)=1,$ $l\in {\mathbb N}$ and $u(0)=0;$
cf. \cite[pp. 186--187]{dud} for more details.

The reflection properties of multidimensional $Z$-transform can be also simply clarified; cf. \cite[p. 181]{dud}. Many other properties of multidimensional $Z$-transform of scalar-valued sequences, like the
multiplication property, the initial value theorem and the property formulated in \cite[Theorem 2.1.6]{gregor}, can be simply extended for the vector-valued sequences; also, it seems very plausible that the Parseval equality can be extended for the sequences with values in complex Hilbert spaces; cf. \cite[p. 182]{dud} and \cite[Lemma 6.8.1]{gil} for more details.

\section{Discrete convolution products and multidimensional $Z$-transform}\label{disc}

In our previous research studies (cf. \cite{funkcionalne} for more details), we have used the discrete convolution product of sequences $(a_{k})_{k\in {\mathbb N}_{0}^{n}}$ and $(b_{k})_{k\in {\mathbb N}_{0}^{n}}$, which is defined by  
$$
(a\ast_{0}b)(k):=\sum_{l\in {\mathbb N}_{0}^{n};l\leq k}a(k-l)b(l),\ k\in {\mathbb N}_{0}^{n};
$$
here, if $a=(a_{1},...,a_{n})\in {\mathbb R}^{n}$ and $b=(b_{1},...,b_{n})\in {\mathbb R}^{n}$, then we write $a\leq b$ if $a_{i}\leq b_{i}$ for $1\leq i\leq n.$
It can be simply proved that the convolution product $\ast_{0}$, which is sometimes also called the discrete convolution product of Faltung, is commutative and associative. 
Furthermore, if the sequences $(a_{k})_{k\in {\mathbb N}_{0}^{n}}$ and $(b_{k})_{k\in {\mathbb Z}^{n}}$ are given, then we have used the Weyl convolution product $(a\circ b)(\cdot)$, which is defined by
$$
(a\circ b)(k):=\sum_{l\in {\mathbb Z}^{n};l\leq k}a(k-l)b(l),\quad k\in {\mathbb Z}^{n}.
$$
Under certain assumptions, the following equalities hold true:
\begin{align}\label{trsic}
\bigl( a\ast_{0}b\bigr) \circ c =b\circ (a \circ c)=a \circ (b \circ c);
\end{align}
cf. \cite[Theorem 3.12(ii)-(iii)]{keyantuo} and the equation \eqref{trsic1} below.

Now we would like to introduce the following extension of above-mentioned convolution products and the usual convolution product $\ast$ of sequences $(a_{k})_{k\in {\mathbb Z}^{n}}$ and $(b_{k})_{k\in {\mathbb Z}^{n}}$, which is defined by  
$$
(a\ast b)(k):=\sum_{l\in {\mathbb Z}^{n}}a(k-l)b(l),\ k\in {\mathbb Z}^{n};
$$
the introduced notion also extends the notion considered in \cite[Definition 3.1.1]{gregor} to the vector-valued setting:

\begin{defn}\label{dcp}
Suppose that $a:D' \rightarrow {\mathbb C}$ and $b: D'' \rightarrow X$, where $\emptyset \neq D' \subseteq {\mathbb Z}^{n}$ and $\emptyset \neq D'' \subseteq {\mathbb Z}^{n}$. Define $D :=D'+D''$, ${\rm D}:=(D',D'')$ and
\begin{align}\label{pr}
\bigl(a \ast_{{\rm D}} b\bigr)(k):=\sum_{l\in D'';\, k-l\in D'}a(k-l)b(l) ,\quad k\in D,
\end{align}
provided that for each seminorm  $p\in \circledast$ and $k\in D$ we have $$\sum_{l\in D'';\, k-l\in D'}|a(k-l)| \cdot p(b(l))<+\infty.$$
\end{defn}

In actual fact, we have $a\ast_{0}b=a \ast_{{\rm D}} b$ with ${\rm D}=({\mathbb N}_{0}^{n},{\mathbb N}_{0}^{n})$, $a\circ b=a \ast_{{\rm D}} b$ with ${\rm D}=({\mathbb N}_{0}^{n},{\mathbb N}_{0}^{n})$ and $a\ast b=a \ast_{{\rm D}} b$ with ${\rm D}=({\mathbb Z}^{n},{\mathbb Z}^{n}).$
Since
$$
\sum_{l\in D'';\, k-l\in D'}a(k-l)b(l)=\sum_{l\in D';\, k-l\in D''}a(l)b(k-l),\quad k\in D,
$$
the discrete convolution product $\ast_{{\rm D}}$ of scalar-valued sequences $a:D' \rightarrow {\mathbb C}$ and $b: D'' \rightarrow {\mathbb C}$ is commutative; more precisely, we have $a \ast_{(D',D'')} b=b\ast_{(D'',D')}a.$ We can also simply thansfer the identites established in the equalities \cite[(3.1.3)]{gregor} to the vector-valued setting.

The introduced discrete convolution product $\ast_{{\rm D}}$ is also associative since a relatively simple computation yields the following generalization of formula \eqref{trsic}:
\begin{align}\label{trsic1}
\Bigl( a\ast_{(D',D'')} b\Bigr) \ast_{(D'+D'',D''')}c=a\ast_{(D',D''+D''')} \Bigl( b\ast_{(D'',D''')}c \Bigr),
\end{align}
which holds
under the assumption that all convolution products in \eqref{trsic1} are well-defined. For certain classes of sequences, the formula \eqref{trsic1} can be deduced by using the following convolution theorem for multidimensional vector-valued $Z$-transform; this result extends the statement of \cite[Theorem 3.1.4]{gregor} to the vector-valued setting:

\begin{thm}\label{dsp}
Suppose that $\emptyset \neq D' \subseteq {\mathbb Z}^{n}$, $\emptyset \neq D'' \subseteq {\mathbb Z}^{n},$ $a:D' \rightarrow {\mathbb C}$, $b: D'' \rightarrow X$, $D =D'+D''$, ${\rm D}=(D',D'')$, $\emptyset \neq \Omega \subseteq {\mathbb C}^{n},$
\begin{align}\label{pr1} 
\sum_{s\in D'}|a(s)| \cdot \bigl|z_{1}\bigr|^{-s_{1}}\bigl|z_{2}\bigr|^{-s_{2}}\ldots \bigl| z_{n}\bigr|^{-s_{n}}<+\infty,\quad z\in \Omega ,
\end{align} 
and
\begin{align}\label{pr2} \sum_{l\in D''}p(b(l)) \cdot
\bigl|z_{1}\bigr|^{-l_{1}}\bigl|z_{2}\bigr|^{-l_{2}}\ldots \bigl| z_{n}\bigr|^{-l_{n}}<+\infty ,\quad z\in \Omega.
\end{align}
Then the convolution product $(a\ast_{{\rm D}} b)(\cdot)$, given by \eqref{pr}, is well-defined, the multidimensional $Z$-transform of sequence $((a\ast_{{\rm D}} b)(k))_{k\in D}$ in the region $\Omega$ is well-defined, and we have 
\begin{align}\label{psd1}
F_{a\ast_{{\rm D}} b}(z)=F_{a}(z)F_{b}(z), \quad z\in \Omega.
\end{align} 
\end{thm}

\begin{proof}
Since {\small
\begin{align*} &
\sum_{k\in D}p\Biggl( \sum_{l\in D'';\, k-l\in D'}a(k-l)b(l)\Biggr)\bigl|z_{1}\bigr|^{-k_{1}}\bigl|z_{2}\bigr|^{-k_{2}}\ldots \bigl| z_{n}\bigr|^{-k_{n}}
\\& \leq \sum_{k\in D}\Biggl[ \sum_{l\in D'';\, k-l\in D'}|a(k-l)| \cdot p(b(l)) \Biggr]\bigl|z_{1}\bigr|^{-k_{1}}\bigl|z_{2}\bigr|^{-k_{2}}\ldots \bigl| z_{n}\bigr|^{-k_{n}}
\\& =\Biggl[\sum_{s\in D'}|a(s)| \cdot \bigl|z_{1}\bigr|^{-s_{1}}\bigl|z_{2}\bigr|^{-s_{2}}\ldots \bigl| z_{n}\bigr|^{-s_{n}}\Biggr] \cdot \Biggl[ \sum_{l\in D''}p(b(l)) \cdot
\bigl|z_{1}\bigr|^{-l_{1}}\bigl|z_{2}\bigr|^{-l_{2}}\ldots \bigl| z_{n}\bigr|^{-l_{n}}\Biggr]<+\infty,
\end{align*}}for any $z\in \Omega$, it readily follows that for each seminorm $p\in \circledast$  and $k\in D$ we have $\sum_{l\in D'';\, k-l\in D'}|a(k-l)| \cdot p(b(l))<+\infty, $ so that the convolution product $(a\ast_{{\rm D}} b)(\cdot)$ is well-defined; here we have employed the estimates \eqref{pr1} and \eqref{pr2}.
Also, the multidimensional $Z$-transform of sequence $((a\ast_{{\rm D}} b)(k))_{k\in D}$ in the region $\Omega$ is well-defined. The convolution identity \eqref{psd1} can be deduced similarly, since we have:
\begin{align*} &
F_{a}(z)F_{b}(z)=\Biggl[\sum _{s\in D }a\bigl(s\bigr)z_{1}^{-s_{1}}z_{2}^{-s_{2}}\ldots z_{n}^{-s_{n}}\Biggr] \cdot \Biggl[ \sum _{l\in D'' }b\bigl(l\bigr)z_{1}^{-l_{1}}z_{2}^{-l_{2}}\ldots z_{n}^{-l_{n}}\Biggr]
\\& =\sum _{s\in D }\sum _{l\in D'' } a\bigl(s\bigr)b\bigl(l\bigr)z_{1}^{-(s_{1}+l_{1})}z_{2}^{-(s_{2}+l_{2})}\ldots z_{n}^{-(s_{n}+l_{n})} \\&
=\sum_{k\in D}\Biggl[ \sum_{l\in D'';\, k-l\in D'}a(k-l)b(l)\Biggr] z_{1}^{-k_{1}}z_{2}^{-k_{2}}\ldots z_{n}^{-k_{n}} 
\\& =F_{a\ast_{{\rm D}} b}(z) , \quad z\in \Omega.
\end{align*} 
\end{proof}

The interested reader may try to extend the statement of \cite[Theorem 3.1.5]{gregor} to the vector-valued setting. Also, the notion introduced in \cite[Definition 3.1.3]{gregor} can be simply extended for the vector-valued functions: If $\emptyset \neq D' \subseteq {\mathbb Z}^{n}$, $f : D\rightarrow X$ and $\beta \in {\mathbb Z}^{n},$ then the $\beta$-shift of $f(\cdot),$ denoted by $f_{\beta}(\cdot)$, is defined through
$f_{\beta}(k):=f(k+\beta),$ if $k+\beta \in D'$ and $f_{\beta}(k):=0,$ otherwise.

\subsection{The discrete convolution product $\ast_{{\rm D}}^{l,j}$}\label{ima}

In this subsection, we analyze the following extension of the notion introduced in Definition \ref{dcp}, which can be recovered by plugging $l=n$ and $j=1:$

\begin{defn}\label{dcplj}
Suppose that $1\leq l\leq n,$ $a_{l,n}:={n\choose l},$ $1\leq j\leq a_{l,n}$, $D_{l,j}:=\{j_{1},...,j_{l}\}$ is a fixed subset of $\{1,...,n\}$, and $1\leq j_{1}<...<j_{l}\leq n.$ Let $a:D' \rightarrow {\mathbb C}$ and $b: D'' \rightarrow X$, where $\emptyset \neq D' \subseteq {\mathbb Z}^{l}$ and $\emptyset \neq D'' \subseteq {\mathbb Z}^{n}$. Then we define a sequence $a \ast_{{\rm D}}^{l,j}b$ in the following way:
\begin{align*} &
D\Bigl(a \ast_{{\rm D}}^{l,j}b \Bigr):=\Biggl\{ \bigl(k_{1},...,k_{n}\bigr) \in {\mathbb Z}^{n} : \exists \bigl( w_{j_{1}},...,w_{j_{l}} \bigr)\in {\mathbb Z}^{l} \mbox{ s.t. }\\&
\bigl( k_{j_{1}}-w_{j_{1}},...,k_{j_{l}}-w_{j_{l}} \bigr)\in D'\mbox{ and }
\\&
\Bigl(k_{1},...,k_{j_{1}-1},w_{j_{1}},k_{j_{1}+1},...,k_{j_{2}-1},w_{j_{2}},k_{j_{2}+1},...,k_{j_{l}-1}, w_{j_{l}},k_{j_{l}+1},...,k_{n}\Bigr)\in D''
\Biggr\},
\end{align*}
and
\begin{align*} &
\Bigl(a \ast_{{\rm D}}^{l,j}b \Bigr)\bigl(k_{1},...,k_{n}\bigr):=\sum_{( w_{j_{1}},...,w_{j_{l}} )\in S_{k,{\rm D},l,j}}a\bigl( k_{j_{1}}-w_{j_{1}},...,k_{j_{l}}-w_{j_{l}} \bigr)
\\& \times  b\Bigl(k_{1},...,k_{j_{1}-1},w_{j_{1}},k_{j_{1}+1},...,k_{j_{2}-1},w_{j_{2}},k_{j_{2}+1},...,k_{j_{l}-1}, w_{j_{l}},k_{j_{l}+1},...,k_{n}\Bigr) ,
\end{align*}
whenever $(k_{1},...,k_{n}) \in D(a \ast_{{\rm D}}^{l,j}b ) $ and the above series absolutely converges; here and hereafter, for any $k=(k_{1},...,k_{n}) \in {\mathbb Z}^{n},$ $S_{k,{\rm D},l,j}$ denotes the set of all tuples $( w_{j_{1}},...,w_{j_{l}}  )\in {\mathbb Z}^{l}$ such that the property stated in the definition of $D(a \ast_{{\rm D}}^{l,j}b)$ holds.

Finally, we set $a_{0,0}:=1$ and $a \ast_{{\rm D}}^{0,1}b:=b.$
\end{defn}

It is worth noting that the case $D''={\mathbb Z}^{n}$ is quite exceptional since, in this case, we have $S_{k,{\rm D},l,j}={\mathbb Z}^{l}$ for every $k=(k_{1},...,k_{n}) \in {\mathbb Z}^{n}$, $1\leq l\leq n$ and $1\leq j\leq a_{l,n}.$ 

The proof of following result is very similar to the proof of Theorem \ref{dsp} and therefore omitted:

\begin{thm}\label{dsplj}
Suppose that $\emptyset \neq \Omega \subseteq {\mathbb C}^{n},$ $1\leq l\leq n,$ $1\leq j\leq a_{l,n}$, $D_{l,j}=\{j_{1},...,j_{l}\}$ is a fixed subset of $\{1,...,n\}$, and $1\leq j_{1}<...<j_{l}\leq n.$ Let $a:D' \rightarrow {\mathbb C}$ and $b: D'' \rightarrow X$, where $\emptyset \neq D' \subseteq {\mathbb Z}^{l}$ and $\emptyset \neq D'' \subseteq {\mathbb Z}^{n}$, and let $D:=
D (a \ast_{{\rm D}}^{l,j}b  )\neq \emptyset.$  If 
\begin{align}\label{pr1lj} 
\sum_{s\in D'}|a(s)| \cdot \bigl|z_{j_{1}}\bigr|^{-s_{j_{1}}}\bigl|z_{j_{2}}\bigr|^{-s_{j_{2}}}\ldots \bigl| z_{j_{l}}\bigr|^{-s_{j_{l}}}<+\infty,\quad z=\bigl(z_{1},...,z_{n}\bigr)\in \Omega 
\end{align} 
and
\begin{align}\label{pr2lj} 
\sum_{l\in D''}p(b(l)) \cdot
\bigl|z_{1}\bigr|^{-l_{1}}\bigl|z_{2}\bigr|^{-l_{2}}\ldots \bigl| z_{n}\bigr|^{-l_{n}}<+\infty ,\quad z\in \Omega ,
\end{align}
then the convolution product $(a\ast_{{\rm D}}^{l,j} b)(\cdot)$ is well-defined, the multidimensional $Z$-transform of sequence $((a\ast_{{\rm D}}^{l,j} b)(k))_{k\in D}$ in the region $\Omega$ is well-defined, and we have 
\begin{align}\label{psd1}
F_{a\ast_{{\rm D}}^{l,j} b}(z)=F_{a}\bigl(z_{j_{1}},...,z_{j_{l}} \bigr)F_{b}\bigl(z_{1},...,z_{n}\bigr), \quad z=\bigl(z_{1},...,z_{n}\bigr)\in \Omega.
\end{align} 
\end{thm}

\section{Applications of multidimensional vector-valued $Z$-transform to abstract Volterra difference equations with multiple variables}\label{sire}

In this section, we will present various applications of multidimensional $Z$-transform to the abstract fractional partial difference equations and abstract Volterra difference equations depending on several variables. We will divide the material of this section into three separate subsections.

\subsection{Abstract linear difference equations depending on several variables}\label{subs2}

This subsection is devoted to the study of the following abstract linear difference equation on ${\mathbb Z}^{n}:$
\begin{align}\label{afro}
\sum_{j\in D}A_{j}u(k+j)=Cf(k),\quad k\in {\mathbb Z}^{n},
\end{align}
where $\emptyset \neq D\subseteq {\mathbb Z}^{n}$, $D$ is finite, $C\in L(X)$, $A_{j}$ is a linear operator on $X$ ($j\in D$) and $f: {\mathbb Z}^{n} \rightarrow X$, as well as the following abstract linear difference equation on ${\mathbb N}_{0}^{n}:$
\begin{align}\label{afro1}
\sum_{j\in D}A_{j}u(k+j)=Cf(k),\quad k\in {\mathbb N}_{0}^{n},
\end{align}
where $\emptyset \neq D\subseteq {\mathbb N}_{0}^{n}$, $D$ is finite, $C\in L(X)$, $A_{j}$ is a linear operator on $X$ ($j\in D$)  and $f: {\mathbb N}_{0}^{n} \rightarrow X,$ subjected with the initial conditions of form 
\begin{align}\label{afro2}
f(k)=0,\quad k\in \bigcup_{j\in D}\Bigl[{\mathbb N}_{0}^{n} \setminus \bigl(j+{\mathbb N}_{0}^{n}\bigr) \Bigr].
\end{align}

By a solution of \eqref{afro}, resp. \eqref{afro1}-\eqref{afro2}, we mean any sequence $u : {\mathbb Z}^{n} \rightarrow X$, resp. $u : {\mathbb N}_{0}^{n} \rightarrow X,$ such that
$u(k+j)\in D(A_{j})$ for all $j\in D$, $k\in {\mathbb Z}^{n}$ and \eqref{afro} holds for all $k\in {\mathbb Z}^{n}$, resp. 
$u(k+j)\in D(A_{j})$ for all $j\in D$, $k\in {\mathbb N}_{0}^{n}$ and \eqref{afro1}-\eqref{afro2} hold for all $k\in {\mathbb N}_{0}^{n}.$

Let us consider first the problem \eqref{afro}. Applying the multidimensional $Z$-transform to the both sides of \eqref{afro}, provided that Proposition \ref{manulakis} can be applied, we get 
$$
\sum_{j\in D}A_{j}\bigl[ z^{j} F_{u}(z)\bigr]=CF_{f}(z),\quad z\in \Omega ,\mbox{ i.e., }\Biggl[\sum_{j\in D}z^{j}A_{j}\Biggr] \cdot F_{u}(z)=CF_{f}(z),\quad z\in \Omega ,
$$
and
\begin{align}\label{ge0}
F_{u}(z)=\Biggl[\sum_{j\in D}z^{j}A_{j}\Biggr]^{-1}C \cdot F_{f}(z),\quad z\in \Omega .
\end{align}
Now, if Theorem \ref{dsp} and Theorem \ref{seka} can be applied, the equation \eqref{ge0} indicates us that we should have $u(k)= (G\ast_{(D',{\mathbb Z}^{n})} f)(k),$ $k\in {\mathbb Z}^{n},$ where $\emptyset \neq D' \subseteq {\mathbb Z}^{n}$ and
\begin{align}\notag
G(k)x=\frac{1}{(2\pi i)^{n}}&\oint_{|z_{1}|=r_{1}}\cdot ... \cdot \oint_{|z_{n}|=r_{n}}z_{1}^{k_{1}-1}\cdot ... \cdot z_{n}^{k_{n}-1}
\\\label{ge} & \times
\Biggl[\sum_{j\in D}z_{1}^{j_{1}}\cdot ... \cdot z_{n}^{j_{n}}A_{(j_{1},....,j_{n})}\Biggr]^{-1}Cx\, dz_{1}\, ...\, dz_{n},
\end{align}
for any $x\in X$ and $k=(k_{1},...,k_{n})\in D'.$  

Now we will state and prove the following results on the well-posedness of problem \eqref{afro}:

\begin{thm}\label{moves}
Suppose that $X$ is a barreled \emph{SCLCS}, $A_{j}$ is a closed linear operator on $X$ for all $j\in D,$ $f: {\mathbb Z}^{n} \rightarrow X$, $\emptyset \neq \Omega \subseteq {\mathbb C}^{n}$, $r_{i}>0$ for $1\leq i\leq n$ and $\Omega$ contains an open neighborhood of the set $\{z_{1}\in {\mathbb C} : |z_{1}|=r_{1}\} \times ... \times \{z_{n}\in {\mathbb C} : |z_{n}|=r_{n}\}$ in ${\mathbb C}^{n}.$ 
Suppose, further, that there exists a set $\emptyset \neq D' \subseteq {\mathbb Z}^{n}$ such that the following conditions hold:
\begin{itemize}
\item[(i)] $[\sum_{j\in D}z^{j}A_{j}]^{-1}C\in L(X),$ $z\in \Omega$ and the mapping $z\mapsto [\sum_{j\in D}z^{j}A_{j}]^{-1}Cx,$ $z\in \Omega$ is holomorphic for every fixed element $x\in X.$
\item[(ii)] For every $k=(k_{1},...,k_{n})\in {\mathbb Z}^{n}$, $p\in \circledast $ and $j'\in D$ we have
\begin{align} & \label{afro5}
\sum_{l\in {\mathbb Z}^{n}; k-l\in D'}r_{1}^{k_{1}-l_{1}}\cdot ... \cdot r_{n}^{k_{n}-l_{n}} \\ \notag &  \times \sup_{|w_{1}|=r_{1};...;|w_{n}|=r_{n}} p\Biggl( \Biggl[\sum_{j\in D}w_{1}^{j_{1}}\cdot ... \cdot w_{n}^{j_{n}}A_{(j_{1},....,j_{n})}\Biggr]^{-1}Cf(l)\Biggr) <+\infty
\end{align}
and
\begin{align} & \notag
\sum_{l\in {\mathbb Z}^{n}; k-l\in D'}r_{1}^{k_{1}-l_{1}}\cdot ... \cdot r_{n}^{k_{n}-l_{n}}
\\\label{afro51} & \times \sup_{|w_{1}|=r_{1};...;|w_{n}|=r_{n}} p\Biggl( A_{j'}\Biggl[\sum_{j\in D}w_{1}^{j_{1}}\cdot ... \cdot w_{n}^{j_{n}}A_{(j_{1},....,j_{n})}\Biggr]^{-1}Cf(l)\Biggr)<+\infty .
\end{align}
\item[(iii)] The multidimensional $Z$-transform of sequence $(f(k))_{k\in {\mathbb Z}^{n}}$ exists in the region $\Omega ,$ the multidimensional $Z$-transform of sequence $(G(k))_{k\in D'}$, where $G(k)$ is defined by \eqref{ge} for all $k\in D',$ exists in the region $\Omega ,$ and the multidimensional $Z$-transform of sequence $(A_{j}G(k))_{k\in D'}$ in $L(X)$ exists in the region $\Omega$ for all $j\in D.$
\end{itemize} 
Then the problem \eqref{afro} has a solution given by $u(k)= (G\ast_{(D',{\mathbb Z}^{n})} f)(k),$ $k\in {\mathbb Z}^{n}.$ 
\end{thm} 

\begin{proof} 
Since we have assumed that $X$ is a barreled SCLCS and (i) holds, it follows that $L(X)$ is an SCLCS and the mapping $z\mapsto [\sum_{j\in D}z^{j}A_{j}]^{-1}C\in L(X),$ $z\in \Omega$ is holomorphic; cf. \cite{FKP} for more details.
Using \eqref{afro5} and \eqref{afro51}, it follows that the sequence
\begin{align*}
u(k)&= \bigl(G\ast_{(D',{\mathbb Z}^{n})} f\bigr)(k)=\sum_{l\in {\mathbb Z}^{n};\, k-l\in D'}G(k-l)f(l)
\\& =\sum_{l\in {\mathbb Z}^{n};\, k-l\in D'}\frac{1}{(2\pi i)^{n}}\oint_{|z_{1}|=r_{1}}\cdot ... \cdot \oint_{|z_{n}|=r_{n}}z_{1}^{k_{1}-1-l_{1}}\cdot ... \cdot z_{n}^{k_{n}-1-l_{n}} \\ & \times
\Biggl[\sum_{j\in D}z_{1}^{j_{1}}\cdot ... \cdot z_{n}^{j_{n}}A_{(j_{1},....,j_{n})}\Biggr]^{-1}Cf(l)\, dz_{1}\, ...\, dz_{n} ,\ k=(k_{1},...,k_{n})\in {\mathbb Z}^{n},
\end{align*}
is well-defined as well as that for each $j\in D$ we have $R(u)\subseteq D(A_{j})$. By (iii), Theorem \ref{seka} and the proof of Theorem \ref{dsp}, we obtain that $F_{u}(z)$ exists for all $z\in \Omega$ and 
\begin{align}\label{prcko}
F_{u}(z)=\Biggl( \sum_{j\in D}z^{j}A_{j}\Biggr)^{-1}CF_{f}(z),\quad z\in \Omega.
\end{align}  
Since the multidimensional $Z$-transform of sequence $(f(k))_{k\in {\mathbb Z}^{n}}$ exists in the region $\Omega  $ and the multidimensional $Z$-transform of sequence $(A_{j}G(k))_{k\in D'}$ in $L(X)$ exists in the region $\Omega$ for all $j\in D,$ it follows that the multidimensional $Z$-transform of sequence $(A_{j}u(k))_{k\in {\mathbb Z}^{n}}$ exists in the region $\Omega  $ for all $j\in D.$ By the argumentation contained in the proof of Proposition \ref{manulakis}, we get that
$$
\sum_{k\in {\mathbb Z}^{n}}p\Bigl(A_{j}u(k+j) \Bigr)
\bigl|z_{1}\bigr|^{-k_{1}}\bigl|z_{2}\bigr|^{-k_{2}}\ldots \bigl| z_{n}\bigr|^{-k_{n}}<+\infty
$$
for all $z\in {\Omega}$ and $j\in D .$ Therefore, the multidimensional $Z$-transform of sequence $(\sum_{j\in D}A_{j}u(j+k))_{k\in {\mathbb Z}^{n}}$ exists in the region $\Omega $ due to Proposition \ref{linear}, and we have
\begin{align*} 
F_{\sum_{j\in D}A_{j}u(j+\cdot)}(z)=\sum_{j\in D}F_{A_{j}u(j+\cdot)}(z),\quad z\in \Omega.
\end{align*}
Since the operator $A_{j}$ is closed on $X$ for all $j\in D,$ we simply get from the equality \eqref{prcko} and Proposition \ref{manulakis} that
\begin{align*} & F_{\sum_{j\in D}A_{j}u(j+\cdot)}(z)=\sum_{j\in D}F_{A_{j}u(j+\cdot)}(z)=\sum_{j\in D}z^{j}F_{A_{j}u}(z)=\sum_{j\in D}z^{j}A_{j}F_{u}(z)
\\&=\Biggl[\sum_{j\in D}z^{j}A_{j}\Biggr] \cdot  \Biggl( \sum_{j\in D}z^{j}A_{j}\Biggr)^{-1}CF_{f}(z)=CF_{f}(z)=F_{Cf}(z),\ z\in \Omega.
\end{align*}
By the uniqueness of multidimensional vector-valued $Z$-transform, we get that \eqref{afro} holds, which completes the proof.
\end{proof}

We can similarly deduce the following result:

\begin{thm}\label{moves1}
Suppose that $X$ is a barreled \emph{SCLCS}, $A_{j}$ is a closed linear operator on $X$ for all $j\in D,$ $f: {\mathbb N}_{0}^{n} \rightarrow X$, $\emptyset \neq \Omega \subseteq {\mathbb C}^{n}$, $r_{i}>0$ for $1\leq i\leq n$ and $\Omega$ contains an open neighborhood of the set $\{z_{1}\in {\mathbb C} : |z_{1}|=r_{1}\} \times ... \times \{z_{n}\in {\mathbb C} : |z_{n}|=r_{n}\}$ in ${\mathbb C}^{n}.$ 
Suppose, further, that there exists a set $\emptyset \neq D' \subseteq {\mathbb N}_{0}^{n}$ such that $0\in D'$ and the following conditions hold:
\begin{itemize}
\item[(i)] $[\sum_{j\in D}z^{j}A_{j}]^{-1}C\in L(X),$ $z\in \Omega$ and the mapping $z\mapsto [\sum_{j\in D}z^{j}A_{j}]^{-1}Cx,$ $z\in \Omega$ is holomorphic for every fixed element $x\in X.$
\item[(ii)] For every $k=(k_{1},...,k_{n})\in {\mathbb N}_{0}^{n}$, $p\in \circledast $ and $j'\in D$ we have
\begin{align*} & 
\sum_{l\in {\mathbb N}_{0}^{n}; k-l\in D'}r_{1}^{k_{1}-l_{1}}\cdot ... \cdot r_{n}^{k_{n}-l_{n}} \\ \notag &  \times \sup_{|w_{1}|=r_{1};...;|w_{n}|=r_{n}} p\Biggl( \Biggl[\sum_{j\in D}w_{1}^{j_{1}}\cdot ... \cdot w_{n}^{j_{n}}A_{(j_{1},....,j_{n})}\Biggr]^{-1}Cf(l)\Biggr) <+\infty
\end{align*}
and
\begin{align*} &
\sum_{l\in {\mathbb N}_{0}^{n}; k-l\in D'}r_{1}^{k_{1}-l_{1}}\cdot ... \cdot r_{n}^{k_{n}-l_{n}}
\\& \times \sup_{|w_{1}|=r_{1};...;|w_{n}|=r_{n}} p\Biggl( A_{j'}\Biggl[\sum_{j\in D}w_{1}^{j_{1}}\cdot ... \cdot w_{n}^{j_{n}}A_{(j_{1},....,j_{n})}\Biggr]^{-1}Cf(l)\Biggr) <+\infty .
\end{align*}
\item[(iii)] The multidimensional $Z$-transform of sequence $(f(k))_{k\in {\mathbb N}_{0}^{n}}$ exists in the region $\Omega ,$ the multidimensional $Z$-transform of sequence $(G(k))_{k\in D'}$  in $L(X)$, where $G(k)$ is defined by \eqref{ge} for all $k\in D',$ exists in the region $\Omega ,$ and the multidimensional $Z$-transform of sequence $(A_{j}G(k))_{k\in D'}$ in $L(X)$ exists in the region $\Omega$ for all $j\in D.$
\end{itemize} 
Then the problem \eqref{afro1}-\eqref{afro2} has a solution given by $u(k)= (G\ast_{(D',{\mathbb N}_{0}^{n})} f)(k),$ $k\in {\mathbb N}_{0}^{n}.$ 
\end{thm} 

\begin{rem}\label{pokazna}
Suppose that the mapping $z\mapsto [\sum_{j\in D}z^{j}A_{j}]^{-1}Cx,$ $z\in \Omega$ is continuous for every fixed element $x\in X,$ the mapping $z\mapsto A_{j'}[\sum_{j\in D}z^{j}A_{j}]^{-1}Cx,$ $z\in \Omega$ is holomorphic for every fixed element $x\in X$ and tuple $j'\in D,$ $C\in L(X)$ is injective and commutes with $A_{j}$ for all $j\in D.$ Then we can use the Hilbert resolvent equality in order to prove that
the mapping $z\mapsto [\sum_{j\in D}z^{j}A_{j}]^{-1}Cx,$ $z\in \Omega$ is holomorphic for every fixed element $x\in X;$ cf. also \cite[Lemma 2.6.3]{FKP} and the proofs of \cite[Theorem 1--Theorem 2]{fcaa2025}.
\end{rem}

Now we would like to present the following illustrative application of Theorem \ref{moves} and Theorem \ref{moves1}:

\begin{example}\label{zan}
A great number of concrete examples can be used to provide applications of Theorem \ref{moves} in the case that $D'={\mathbb N}_{0}^{n}$, the sequence $(f(k))_{k\in {\mathbb Z}^{n}}$ rapidly decays at infinity, the operator $(\sum_{j\in D}z^{j}A_{j})^{-1}C$ exists in the set $\Lambda =\{(z_{1},...,z_{n})\in {\mathbb C}^{n}: r_{1}'\leq |z_{1}|\leq r_{1}'';...;r_{n}'\leq |z_{n}|\leq r_{n}''\}$, where $r_{i}''>r_{i}>r_{i}'>0$ for $1\leq i\leq n$, the mapping $z\mapsto (\sum_{j\in D}z^{j}A_{j})^{-1}C\in L(X)$ is holomorphic in $\Omega=\Lambda^{\circ}$ and continuous on the boundary of $\Lambda$. Assuming, for example, that for each seminorm $p\in \circledast$ there exist a real number $M>0$ and a seminorm $q\in \circledast$ such that 
\begin{align}\label{re}
p\Biggl(\Biggl(\sum_{j\in D}z^{j}A_{j}\Biggr)^{-1}Cx\Biggr)+\sum_{j'\in D}p\Biggl(A_{j'}\Biggl(\sum_{j\in D}z^{j}A_{j}\Biggr)^{-1}Cx\Biggr)\leq Mq(x),\ x\in X, \ z\in \Lambda
\end{align}
and   
$$
f(k)=e^{-k_{1}^{2}-...-k_{n}^{2}},\quad k=\bigl(k_{1},...,k_{n}\bigr)\in {\mathbb Z}^{n},
$$ 
then all requirements for applying the above-mentioned result can be simply verified. Let us only note that we can use the Cauchy theorem to show that the sequence $G(\cdot)$ can be defined by the integration over the curves $|z_i|=r_{i'}$ for $1\leq i\leq n,$ as well as that the  estimate \eqref{re} shows that for each seminorm $p\in \circledast$ there exists a real constant $M_{p}>0$ such that 
\begin{align*} 
p\Bigl(G\bigl(k_{1},...,k_{n}\bigr)\Bigr)& +\sum_{j\in D}p\Bigl(A_{j}G\bigl(k_{1},...,k_{n}\bigr)\Bigr)
\\& \leq M_{p}\Bigl[\bigl(r_{1}'\bigr)^{k_{1}}\cdot ... \cdot \bigl(r_{n}'\bigr)^{k_{n}}\Bigr],\quad k=\bigl(k_{1},...,k_{n}\bigr)\in {\mathbb N}_{0}^{n},
\end{align*}
so that the the multidimensional $Z$-transform of sequence $(G(k))_{k\in {\mathbb N}_{0}^{n}}$  in $L(X)$ and the multidimensional $Z$-transform of sequence $(A_{j}G(k))_{k\in {\mathbb N}_{0}^{n}}$ in $L(X)$ exists in the region $\Omega$ for all $j\in D;$ cf. also the computation given directly after the estimate \eqref{est1}.

Similarly we can apply Theorem \ref{moves1} in the case that $D'={\mathbb N}_{0}^{n}$ and the sequence $f(\cdot)$ satisfies the estimate \eqref{est1}.
\end{example}

Concerning the uniqueness of solutions of problem \eqref{afro}, we will first introduce the following space{\small
\begin{align*} &
uniq(X;\eqref{afro}):=\Bigl\{  u : {\mathbb Z}^{n} \rightarrow X : \mbox{ exists an open non-empty set }\Omega \subseteq {\mathbb C}^{n}\mbox{ such that }F_{u}(z)\\& \mbox{ and }F_{A_{j}u}(z)  \mbox{ are well-defined for all } j\in D\mbox{ and }z\in \Omega \Bigr\}.
\end{align*}}Using Proposition \ref{linear}, it readily follows that $uniq(X;\eqref{afro})$ is a vector space with the usual operations. Now we will state and prove the following result:

\begin{thm}\label{iniq}
Suppose that $A_{j}$ is a closed linear operator ($j\in D$) and there exists an open non-empty set  $\Omega \subseteq {\mathbb C}^{n}$ such that  the operator $\sum_{j\in D}z^{j}A_{j}$ is injective for all $z=(z_{1},...,z_{n})\in \Omega .$ Then there exists a unique solution of problem \eqref{afro} which belongs to the space $uniq(X;\eqref{afro}).$  
\end{thm}

\begin{proof}
Suppose that $u_{1}(\cdot)$ and $u_{2}(\cdot)$ are solutions of problem \eqref{afro} which belongs to the space $uniq(X;\eqref{afro}).$  Then the sequence $u:=u_{1}-u_{2}$ is a solution of problem \eqref{afro} which belongs to the space $uniq(X;\eqref{afro}),$ with $f\equiv 0.$ Due to Proposition \ref{linear}, Proposition \ref{manulakis} and the closedness of operator $A_{j}$ for each $j\in D$, we have
\begin{align*}
F_{\sum_{j\in D}A_{j}u(j+\cdot)}(z)&=\sum_{j\in D}F_{A_{j}u(j+\cdot)}(z)=\sum_{j\in D}A_{j}F_{u(j+\cdot)}(z)=\sum_{j\in D}z^{j}A_{j}F_{u}(z)=0, 
\end{align*}
for all $z\in \Omega$. Since the operator $\sum_{j\in D}z^{j}A_{j}$ is injective for all $z\in \Omega$, the above implies $F_{u}(z)=0,$ $z\in \Omega.$ Hence, $u\equiv 0,$ as claimed. 
\end{proof}

We can similarly formulate an analogue of Theorem \ref{iniq} for the problem [\eqref{afro1}-\eqref{afro2}]. An illustrative application  is given below:

\begin{example}\label{njas}
Suppose that $n=2,$ $D=\{(1,0),(0,0)\},$ $A_{(0,0)}=-{\rm I}$ and $A_{(1,0)}=A$, where $A$ is a closed linear operator on $X$ which does not posssess any eigenvalue in a neighborhood of zero. Then Theorem \ref{iniq} can be applied with $\Omega=\{(z_{1},z_{2}) \in {\mathbb C}^{2} : |z_{1}|>r_{1},\ |z_{2}|>r_{2}\}$ for sufiiciently large real numbers $r_{1}>0$ and $r_{2}>0$. In conclusion, we obtain that there exists a unique solution of the abstract partial difference equation
$$
Au\bigl(k_{1}+1,k_{2}+1\bigr)-u\bigl(k_{1},k_{2}\bigr)=f\bigl(k_{1},k_{2}\bigr),\quad k= \bigl(k_{1} ,k_{2} \bigr)\in {\mathbb Z}^{2}
$$
which belongs to the space $uniq(X;\eqref{afro}).$
\end{example}

\subsection{Abstract multi-term Volterra difference equations on ${\mathbb Z}^{n}$}\label{subs1}

In \cite{bilten}, we have recently analyzed the following abstract multi-term fractional difference equation with generalized Weyl fractional derivatives: 
\begin{align*}
Bu(k)&=A_{1}\bigl(a_{1}\circ u\bigr)(k+k_{1})+... +A_{m}\bigl(a_{m}\circ u\bigr)(k+k_{m})+f(k),\ k\in {\mathbb Z}^{n},
\end{align*}
where $B,\ A_{j}$ are linear operators on $X$, $a_{j} : {\mathbb N}_{0}^{n}\rightarrow {\mathbb C}$, $k_{j}\in {\mathbb Z}^{n}$ ($1\leq j\leq m$) and $f: {\mathbb Z}^{n} \rightarrow X$. The above equation is a special case of the problem \eqref{problem} below, which will be the central point of our investigation in this subsection; see Subsection \ref{disc} for the notion and more details on the discrete convolution product $\ast_{(D',D'')}$.

Of concern is the following abstract multi-term Volterra difference equation on ${\mathbb Z}^{n}:$
\begin{align}\notag & 
Bu(k)+\sum_{j_{1}=0}^{l_{1}}A_{1,j_{1}}\Bigl( a_{1,j_{1}} \ast_{(D_{1},{\mathbb Z}^{n})}u \Bigr)\bigl(k+k_{1,j_{1}}\bigr)
\\\label{problem}& +...+\sum_{j_{s}=0}^{l_{s}}A_{s,j_{s}}\Bigl( a_{s,j_{s}} \ast_{(D_{s},{\mathbb Z}^{n})} u\Bigr)\bigl(k+k_{s,j_{s}}\bigr)=f(k),\ k\in {\mathbb Z}^{n};
\end{align}
in this subsection, we will always assume that condition (C1) clarified in the introductory part holds.

We will use the following notion of strong solution to \eqref{problem}; the notion of mild solution to \eqref{problem} can be also introduced and analyzed:

\begin{defn}\label{diuke}
It is said that a sequence $u: {\mathbb Z}^{n}\rightarrow X$ is
a strong solution of problem \eqref{problem} if $u(k)\in D(W)$ for every $k\in {\mathbb Z}^{n} $ and for every operator $W$ from the set $\{B,A_{1,0},...,A_{1,l_{1}},..., A_{s,0},...,A_{s,l_{s}}\},$ the values of
$( a_{1,j_{1}} \ast_{(D_{1},{\mathbb Z}^{n})}A_{1,j_{1}}u )(k+k_{1,j_{1}}) $, ..., $( a_{s,j_{s}} \ast_{(D_{s},{\mathbb Z}^{n})}A_{s,j_{s}}u )(k+k_{s,j_{s}}) $ are well-defined for all $k\in {\mathbb Z}^{n},$ and \eqref{problem} holds for all $k\in {\mathbb Z}^{n}.$
\end{defn}

Applying multidimensional $Z$-transform to both sides of \eqref{problem}, we get:{\scriptsize
\begin{align*}
BF_{u}(z)&+\sum_{j_{1}=0}^{l_{1}}z^{k_{1,j_{1}}}A_{1,j_{1}}F_{a_{1,j_{1}} \ast_{(D_{1},{\mathbb Z}^{n})u}}(z)
\\& +....+\sum_{j_{s}=0}^{l_{s}}A_{s,j_{s}}z^{k_{s,j_{s}}}A_{s,j_{s}}F_{a_{s,j_{s}} \ast_{(D_{1},{\mathbb Z}^{n})u}}(z)=CF_{f}(z),\quad z\in \Omega ,
\end{align*}}i.e.,{\scriptsize
\begin{align*}
BF_{u}(z)&+\sum_{j_{1}=0}^{l_{1}}z^{k_{1,j_{1}}}F_{a_{1,j_{1}}}(z)A_{1,j_{1}}F_{u}(z)  
\\& +....+\sum_{j_{s}=0}^{l_{s}}A_{s,j_{s}}z^{k_{s,j_{s}}}F_{a_{s,j_{s}}}(z)A_{s,j_{s}}F_{u}(z)=CF_{f}(z),\quad z\in \Omega ,
\end{align*}}or equivalently{\scriptsize
\begin{align*}
\Biggl[B &+\sum_{j_{1}=0}^{l_{1}}z^{k_{1,j_{1}}}F_{a_{1,j_{1}}}(z)A_{1,j_{1}} 
\\&+....+\sum_{j_{s}=0}^{l_{s}}A_{s,j_{s}}z^{k_{s,j_{s}}}F_{a_{s,j_{s}}}(z)A_{s,j_{s}}\Biggr] \cdot F_{u}(z)=CF_{f}(z),\quad z\in \Omega .
\end{align*}}Under certain assumptions, the above implies{\scriptsize
\begin{align*}
 F_{u}(z)=\Biggl[B +\sum_{j_{1}=0}^{l_{1}}z^{k_{1,j_{1}}}F_{a_{1,j_{1}}}(z)A_{1,j_{1}} 
+....+\sum_{j_{s}=0}^{l_{s}}A_{s,j_{s}}z^{k_{s,j_{s}}}F_{a_{s,j_{s}}}(z)A_{s,j_{s}}\Biggr]^{-1}CF_{f}(z),
\end{align*}}for any $z\in \Omega , $
and $
u(k)= (R\ast_{(D',{\mathbb Z}^{n})} f)(k),$ $k\in {\mathbb Z}^{n},$ where $\emptyset \neq D' \subseteq {\mathbb Z}^{n}$ and{\scriptsize
\begin{align}\notag &
R(k)x=\frac{1}{(2\pi i)^{n}}\oint_{|z_{1}|=r_{1}}\cdot ... \cdot \oint_{|z_{n}|=r_{n}}z_{1}^{k_{1}-1}\cdot ... \cdot z_{n}^{k_{n}-1}
\\\label{gevolt} & \times
\Biggl[B +\sum_{j_{1}=0}^{l_{1}}z^{k_{1,j_{1}}}F_{a_{1,j_{1}}}(z)A_{1,j_{1}} 
+....+\sum_{j_{s}=0}^{l_{s}}A_{s,j_{s}}z^{k_{s,j_{s}}}F_{a_{s,j_{s}}}(z)A_{s,j_{s}}\Biggr]^{-1}Cx\, dz_{1}\, ...\, dz_{n},
\end{align}}for any $x\in X$ and $k=(k_{1},...,k_{n})\in D';$ we have employed Proposition \ref{manulakis}, Theorem \ref{dsp} and Theorem \ref{seka} in the above computation.

The proof of following result is very similar to the proof of Theorem \ref{moves} and therefore omitted:

\begin{thm}\label{movesvolt}
Suppose that $X$ is a barreled \emph{SCLCS}, \emph{(C1)} holds, $f: {\mathbb Z}^{n} \rightarrow X$, $\emptyset \neq \Omega \subseteq {\mathbb C}^{n}$, $r_{i}>0$ for $1\leq i\leq n$ and $\Omega$ contains an open neighborhood of the set $\{z_{1}\in {\mathbb C} : |z_{1}|=r_{1}\} \times ... \times \{z_{n}\in {\mathbb C} : |z_{n}|=r_{n}\}$ in ${\mathbb C}^{n}.$ 
Suppose, further, that there exists a set $\emptyset \neq D' \subseteq {\mathbb Z}^{n}$ such that the following conditions hold:
\begin{itemize}
\item[(i)] $[B +\sum_{j_{1}=0}^{l_{1}}z^{k_{1,j_{1}}}F_{a_{1,j_{1}}}(z)A_{1,j_{1}} 
+....+\sum_{j_{s}=0}^{l_{s}}z^{k_{s,j_{s}}}F_{a_{s,j_{s}}}(z)A_{s,j_{s}}]^{-1}C\in L(X),$ $z\in \Omega$ and the mapping\\ $z\mapsto [B +\sum_{j_{1}=0}^{l_{1}}z^{k_{1,j_{1}}}F_{a_{1,j_{1}}}(z)A_{1,j_{1}} 
+....+\sum_{j_{s}=0}^{l_{s}}z^{k_{s,j_{s}}}F_{a_{s,j_{s}}}(z)A_{s,j_{s}}]^{-1}Cx,$ $z\in \Omega$ is holomorphic for every fixed element $x\in X.$
\item[(ii)] For every $k=(k_{1},...,k_{n})\in {\mathbb Z}^{n}$ and $p\in \circledast $, we have{\scriptsize
\begin{align*} & 
\sum_{l\in {\mathbb Z}^{n}; k-l\in D'}r_{1}^{k_{1}-l_{1}}\cdot ... \cdot r_{n}^{k_{n}-l_{n}}\sup_{|w_{1}|=r_{1};...;|w_{n}|=r_{n}} \\  &  \times  p\Biggl( \Biggl[B +\sum_{j_{1}=0}^{l_{1}}w^{k_{1,j_{1}}}F_{a_{1,j_{1}}}(w)A_{1,j_{1}} 
+....+\sum_{j_{s}=0}^{l_{s}}w^{k_{s,j_{s}}}F_{a_{s,j_{s}}}(w)A_{s,j_{s}}\Biggr]^{-1}Cf(l)\Biggr) <+\infty
\end{align*}}and{\scriptsize
\begin{align*} & 
\sum_{l\in {\mathbb Z}^{n}; k-l\in D'}r_{1}^{k_{1}-l_{1}}\cdot ... \cdot r_{n}^{k_{n}-l_{n}}\sup_{|w_{1}|=r_{1};...;|w_{n}|=r_{n}}
\\ & \times  p\Biggl(W\Biggl[B +\sum_{j_{1}=0}^{l_{1}}w^{k_{1,j_{1}}}F_{a_{1,j_{1}}}(w)A_{1,j_{1}} 
+....+\sum_{j_{s}=0}^{l_{s}}w^{k_{s,j_{s}}}F_{a_{s,j_{s}}}(w)A_{s,j_{s}}\Biggr]^{-1}Cf(l)\Biggr)<+\infty ,
\end{align*}}for every operator $W$ from the set $\{B,A_{1,0},...,A_{1,l_{1}},..., A_{s,0},...,A_{s,l_{s}}\}.$
\item[(iii)] For every $k=(k_{1},...,k_{n})\in {\mathbb Z}^{n}$, $ i\in {\mathbb N}_{s}$, $0\leq w\leq l_{i}$ and $p\in \circledast $, we have{\scriptsize
\begin{align*} & 
\sum_{l\in {\mathbb Z}^{n}; k-l\in D_{i}}\Bigl| a_{i,s}(k-l)\Bigr| \sum_{m\in {\mathbb Z}^{n}; l-m\in D'}r_{1}^{l_{1}-m_{1}}\cdot ... \cdot r_{n}^{l_{n}-m_{n}}\sup_{|w_{1}|=r_{1};...;|w_{n}|=r_{n}} \\  &    p\Biggl( \Biggl[B +\sum_{j_{1}=0}^{l_{1}}w^{k_{1,j_{1}}}F_{a_{1,j_{1}}}(w)A_{1,j_{1}} 
+....+\sum_{j_{s}=0}^{l_{s}}w^{k_{s,j_{s}}}F_{a_{s,j_{s}}}(w)A_{s,j_{s}}\Biggr]^{-1}Cf(m)\Biggr) <+\infty,
\end{align*}}and{\scriptsize
\begin{align*} & 
\sum_{l\in {\mathbb Z}^{n}; k-l\in D_{i}}\Bigl| a_{i,s}(k-l)\Bigr| \sum_{m\in {\mathbb Z}^{n}; l-m\in D'}r_{1}^{l_{1}-m_{1}}\cdot ... \cdot r_{n}^{l_{n}-m_{n}}\sup_{|w_{1}|=r_{1};...;|w_{n}|=r_{n}} \\  &    p\Biggl(W\Biggl[B +\sum_{j_{1}=0}^{l_{1}}w^{k_{1,j_{1}}}F_{a_{1,j_{1}}}(w)A_{1,j_{1}} 
+....+\sum_{j_{s}=0}^{l_{s}}w^{k_{s,j_{s}}}F_{a_{s,j_{s}}}(w)A_{s,j_{s}}\Biggr]^{-1}Cf(m)\Biggr)<+\infty ,
\end{align*}}for every operator $W$ from the set $\{B,A_{1,0},...,A_{1,l_{1}},..., A_{s,0},...,A_{s,l_{s}}\}.$
\item[(iv)] The multidimensional $Z$-transform of sequence $(f(k))_{k\in {\mathbb Z}^{n}}$ exists in the region $\Omega ,$ he multidimensional $Z$-transform of sequence $(a_{i,w}(k))_{k\in D_{i}}$ exists in the region $\Omega $ for all $ i\in {\mathbb N}_{s}$ and $0\leq w\leq l_{i}$,
the multidimensional $Z$-transform of sequence $(R(k))_{k\in D'}$, where $R(k)$ is defined by \eqref{gevolt} for all $k\in D',$ exists in the region $\Omega ,$ and the multidimensional $Z$-transform of sequence $(WR(k))_{k\in D'}$ in $L(X)$ exists in the region $\Omega ,$ for every operator $W$ from the set $\{B,A_{1,0},...,A_{1,l_{1}},..., A_{s,0},...,A_{s,l_{s}}\}.$
\end{itemize} 
Then the problem \eqref{afro} has a solution given by $u(k)= (R\ast_{(D',{\mathbb Z}^{n})} f)(k),$ $k\in {\mathbb Z}^{n}.$ 
\end{thm}

Concerning the uniqueness of solutions to problem \eqref{problem}, our basic assumption will be that exists an open non-empty set $\Omega \subseteq {\mathbb C}^{n}$ such that $F_{a_{i,w}}(z)$ exists for all $z\in \Omega$, $ i\in {\mathbb N}_{s}$ and $0\leq w\leq l_{i}$. If such a set $\Omega$ is already determined, then we introduce the space
$
uniq_{\Omega}(X;\eqref{problem})$ as the set of all sequences $u : {\mathbb Z}^{n} \rightarrow X$ such that $F_{u}(z) $ and $F_{Wu}(z)$ are well-defined for all $z\in \Omega$ and for every operator $W$ from the set $\{B,A_{1,0},...,A_{1,l_{1}},..., A_{s,0},...,A_{s,l_{s}}\}.$
It is clear that $uniq_{\Omega}(X;\eqref{problem})$ is a vector space with the usual operations. 

We can simply deduce the following analogue of Theorem \ref{iniq}:

\begin{thm}\label{iniqvolt}
Suppose that condition \emph{(C1)} holds and there exists an open non-empty set  $\Omega \subseteq {\mathbb C}^{n}$ such that exists an open non-empty set $\Omega \subseteq {\mathbb C}^{n}$ such that $F_{a_{i,w}}(z)$ exists for all $z\in \Omega$, $ i\in {\mathbb N}_{s}$ and $0\leq w\leq l_{i}$. If the operator $B +\sum_{j_{1}=0}^{l_{1}}z^{k_{1,j_{1}}}F_{a_{1,j_{1}}}(z)A_{1,j_{1}} 
+....+\sum_{j_{s}=0}^{l_{s}}z^{k_{s,j_{s}}}F_{a_{s,j_{s}}}(z)A_{s,j_{s}}$ is injective for all $z\in \Omega ,$ then there exists a unique strong solution of problem \eqref{problem} which belongs to the space $uniq_{\Omega}(X;\eqref{problem}).$  
\end{thm}

We can proive many illsutrative applications of Theorem \ref{movesvolt} and Theorem \ref{iniqvolt}; cf. also Example \ref{parkw} below.

\subsection{The abstract fractional difference equations with generalized Weyl $(a,m)$-fractional derivatives}\label{ljeb}

In this subsection, we 
consider the following abstract multi-term fractional difference equation with generalized Weyl $(a,m)$-fractional derivatives:{\small
\begin{align}\label{qad}
A_{s}\Bigl(\Delta_{W,a_{s},m_{s}}u\Bigr)\bigl(k+k_{s}\bigr)+...+A_{1}\Bigl(\Delta_{W,a_{1},m_{1}}u\Bigr)\bigl(k+k_{1}\bigr)+A_{0}u\bigl(k+k_{0}\bigr)=Cf(k),\ k\in {\mathbb Z},
\end{align}}where 
the following condition holds:
\begin{itemize}
\item[(C2)]
$s\in {\mathbb N}$, $A_{1},...,A_{s}$ are closed linear operators on an SCLCS $X$, $k_{0},k_{1},...,k_{s}\in {\mathbb Z},$ $ f: {\mathbb Z} \rightarrow X$ is a given sequence and $\Delta_{W,a,m}u$ is the generalized Weyl $(a,m)$-fractional derivative of $u(\cdot).$
\end{itemize}

First of all, Theorem \ref{dsp} and Proposition \ref{manulakis} together imply that, under certain logical assumptions, we have
$F_{\Delta_{W,a}u}(z)=F_{a}(z) F_{u}(z),$
$z\in \Omega$ and 
\begin{align}\label{cos}
F_{\Delta_{W,a,m}u}(z)=\sum_{j=0}^{m}(-1)^{m-j}\binom{m}{j}z^{j}F_{a}(z)F_{u}(z),\quad z\in \Omega.
\end{align}

Keeping in mind the equality \eqref{cos},
we can perform the vector-valued $Z$-transform to the both sides of \eqref{qad} to obtain that 
\begin{align*}
\Biggl[ \sum_{w=1}^{s}\sum_{j_{w}=0}^{m_{w}}(-1)^{m_{w}-j_{w}}\binom{m_{w}}{j_{w}}z^{k_{w}+j_{w}}F_{a_{w}}(z)A_{w}+z^{k_{0}}A_{0}\Biggr] \cdot F_{u}(z)=CF_{f}(z),\quad z\in \Omega.
\end{align*}
Hence, 
\begin{align*}
F_{u}(z)=\Biggl[ \sum_{w=1}^{s}\sum_{j_{w}=0}^{m_{w}}(-1)^{m_{w}-j_{w}}\binom{m_{w}}{j_{w}}z^{k_{w}+j_{w}}F_{a_{w}}(z)A_{w}+z^{k_{0}}A_{0}\Biggr] ^{-1}C \cdot F_{f}(z),\quad z\in \Omega ,
\end{align*}
and a solution $u :  {\mathbb Z} \rightarrow X$ of problem \eqref{qad} should be given by the following formula:{\small
\begin{align*}
u(k)= \Biggl(Z^{-1}\Biggl[ \Biggl(\sum_{w=1}^{s}\sum_{j_{w}=0}^{m_{w}}(-1)^{m_{w}-j_{w}}\binom{m_{w}}{j_{w}}z^{k_{w}+j_{w}}F_{a_{w}}(z)A_{w}+z^{k_{0}}A_{0}\Biggr) ^{-1}C\Biggr]  \ast_{(D',{\mathbb Z})} f\Biggr)(k),
\end{align*}}
for any $k\in {\mathbb Z},$ where $Z^{-1}$ denotes the inverse $Z$-transform.

We will use the following notion:

\begin{defn}\label{evroliga}
By a strong solution of problem \eqref{qad}, we mean any sequence $u : {\mathbb Z} \rightarrow X$ such that, for every $w\in {\mathbb N}_{s},$ we have $R(u)\subseteq D(A_{w}) \cap D(A_{0})$ as well as that the term $\Delta_{W,a_{w}}(A_{w}u)$ is well-defined and \eqref{qad} holds for all $k\in {\mathbb Z}.$
\end{defn}

Now we can clarify the following result, which can be deduced in a similar way as Theorem \ref{moves}; let us only note that we need condition (iii) here in order to prove that the series which defines the term $\Delta_{W,a_{w}}(A_{w}u)$ converges absolutely ($w\in {\mathbb N}_{s}$):

\begin{thm}\label{movesvoltw}
Suppose that $X$ is a barreled \emph{SCLCS}, \emph{(C2)} holds, $f: {\mathbb Z} \rightarrow X$, $\emptyset \neq \Omega \subseteq {\mathbb C}$, $r>0$ and $\Omega$ contains an open neighborhood of the set $\{z\in {\mathbb C} : |z|=r\} $ in ${\mathbb C}.$ 
Suppose, further, that there exists a set $\emptyset \neq D' \subseteq {\mathbb Z}$ such that the following conditions hold:
\begin{itemize}
\item[(i)] $[\sum_{w=1}^{s}\sum_{j_{w}=0}^{m_{w}}(-1)^{m_{w}-j_{w}}\binom{m_{w}}{j_{w}}z^{k_{w}+j_{w}}F_{a_{w}}(z)A_{w}+z^{k_{0}}A_{0}] ^{-1}C\in L(X),$ $z\in \Omega$ and the mapping $z\mapsto [\sum_{w=1}^{s}\sum_{j_{w}=0}^{m_{w}}(-1)^{m_{w}-j_{w}}\binom{m_{w}}{j_{w}}z^{k_{w}+j_{w}}F_{a_{w}}(z)A_{w}+z^{k_{0}}A_{0}] ^{-1}Cx,$ $z\in \Omega$ is holomorphic for every fixed element $x\in X.$
\item[(ii)] For every $k\in {\mathbb Z}$ and $p\in \circledast $, we have{\scriptsize
\begin{align*} & 
\sum_{l\in {\mathbb Z}; k-l\in D'}r^{k-l}\sup_{|\lambda|=r}   p\Biggl( \Biggl[\sum_{w=1}^{s}\sum_{j_{w}=0}^{m_{w}}(-1)^{m_{w}-j_{w}}\binom{m_{w}}{j_{w}}\lambda^{k_{w}+j_{w}}F_{a_{w}}(\lambda)A_{w}+\lambda^{k_{0}}A_{0}\Biggr]^{-1}Cf(l)\Biggr) <+\infty
\end{align*}}and{\scriptsize
\begin{align*} & 
\sum_{l\in {\mathbb Z}; k-l\in D'}r^{k-l}\sup_{|\lambda|=r} 
p\Biggl(W\Biggl[\sum_{w=1}^{s}\sum_{j_{w}=0}^{m_{w}}(-1)^{m_{w}-j_{w}}\binom{m_{w}}{j_{w}}\lambda^{k_{w}+j_{w}}F_{a_{w}}(\lambda)A_{w}+\lambda^{k_{0}}A_{0}\Biggr]^{-1}Cf(l)\Biggr)<+\infty ,
\end{align*}}for every operator $W$ from the set $\{A_{0},...,A_{s}\}.$
\item[(iii)] For every $k\in {\mathbb Z}$, $w\in {\mathbb N}_{s}$ and $p\in \circledast $, we have{\scriptsize
\begin{align*} & 
\sum_{s=0}^{+\infty}\Bigl| a_{w}(s)\Bigr| \sum_{l\in {\mathbb Z}; k-s-l\in D'}  r^{k-l-s}  \\  &  \times \sup_{|\lambda|=r}  p\Biggl( W\Biggl[\sum_{w=1}^{s}\sum_{j_{w}=0}^{m_{w}}(-1)^{m_{w}-j_{w}}\binom{m_{w}}{j_{w}}\lambda^{k_{w}+j_{w}}F_{a_{w}}(z)A_{w}+\lambda^{k_{0}}A_{0}\Biggr]^{-1}Cf(l)\Biggr) <+\infty,
\end{align*}}for every operator $W$ from the set $\{{\rm I},A_{0},...,A_{s}\}.$
\item[(iv)] The $Z$-transform of sequence $(f(k))_{k\in {\mathbb Z}}$ exists in the region $\Omega ,$ the $Z$-transform of sequence $(a_{w}(k))_{k\in {\mathbb N}_{0}}$ exists in the region $\Omega $ for all $w\in {\mathbb N}_{s}$,
the $Z$-transform of sequence $(S(k))_{k\in D'}$, where $S(k)$ is defined by 
{\small
\begin{align*}\notag &
S(k)x=\frac{1}{ 2\pi i }\oint_{|z |=r } z ^{k -1} 
\Biggl(\sum_{w=1}^{s}\sum_{j_{w}=0}^{m_{w}}(-1)^{m_{w}-j_{w}}\binom{m_{w}}{j_{w}}z^{k_{w}+j_{w}}F_{a_{w}}(z)A_{w}+z^{k_{0}}A_{0}\Biggr) ^{-1}Cx\, dz,
\end{align*}}for any $x\in X$ and $k\in D',$
exists in the region $\Omega ,$ and the multidimensional $Z$-transform of sequence $(WS(k))_{k\in D'}$ in $L(X)$ exists in the region $\Omega ,$ for every operator $W$ from the set $\{A_{0},...,A_{s}\}.$
\end{itemize} 
Then the problem \eqref{qad} has a strong solution given by $u(k)= (S\ast_{(D',{\mathbb Z})} f)(k),$ $k\in {\mathbb Z}.$ 
\end{thm}

Concerning the uniqueness of solutions to problem \eqref{qad}, our basic assumption will be that exists an open non-empty set $\Omega \subseteq {\mathbb C}$ such that $F_{a_{w}}(z)$ exists for all $z\in \Omega$ and $w \in {\mathbb N}_{s}.$ If such a set $\Omega$ is  determined, then we introduce the space
$
uniq_{\Omega}(X;\eqref{qad})$ as the set of all sequences $u : {\mathbb Z} \rightarrow X$ such that $F_{u}(z) $ and $F_{Wu}(z)$ are well-defined for all $z\in \Omega$ and for every operator $W$ from the set $\{ A_{ 0},...,A_{ s } \}.$
It is clear that $uniq_{\Omega}(X;\eqref{qad})$ is a vector space with the usual operations. 

Then the following holds:

\begin{thm}\label{iniqvoltw}
Suppose that condition \emph{(C2)} holds and there exists an open non-empty set  $\Omega \subseteq {\mathbb C}$ such that exists an open non-empty set $\Omega \subseteq {\mathbb C}$ such that $F_{a_{w}}(z)$ exists for all $z\in \Omega$ and $w \in {\mathbb N}_{s}.$ If the operator\\ $\sum_{w=1}^{s}\sum_{j_{w}=0}^{m_{w}}(-1)^{m_{w}-j_{w}}\binom{m_{w}}{j_{w}}z^{k_{w}+j_{w}}F_{a_{w}}(z)A_{w}+z^{k_{0}}A_{0}$ is injective for all $z\in \Omega ,$ then there exists a unique strong solution of problem \eqref{qad} which belongs to the space $uniq_{\Omega}(X;\eqref{qad}).$  
\end{thm}

Now we will illustrate Theorem \ref{movesvoltw} and Theorem \ref{iniqvoltw} with the following illustrative example:

\begin{example}\label{parkw}
Suppose that $(X,\| \cdot \|)$ is a complex Banach space, $s=2$, $m_{1}=1$ and $m_{2}=2 $.
\begin{itemize}
\item[(i)] In our concrete situation, we have
\begin{align*}&
\sum_{w=1}^{s}\sum_{j_{w}=0}^{m_{w}}(-1)^{m_{w}-j_{w}}\binom{m_{w}}{j_{w}}z^{k_{w}+j_{w}}F_{a_{w}}(z)A_{w}+z^{k_{0}}A_{0}
\\&= \Bigl( z^{k_{2}}-2z^{k_{2}+1}+z^{k_{2}+2}\Bigr)F_{a_{2}}(z)A_{2}+\Bigl( z^{k_{1}+1}-z^{k_{1}}\Bigr)F_{a_{1}}(z)A_{1}+z^{k_{0}}A_{0},\quad z\in \Omega.
\end{align*}
Let $r>1$ be a sufficiently large real number such that the operator $$W\Biggl[\Bigl( z^{k_{2}}-2z^{k_{2}+1}+z^{k_{2}+2}\Bigr)F_{a_{2}}(z)A_{2}+\Bigl( z^{k_{1}+1}-z^{k_{1}}\Bigr)F_{a_{1}}(z)A_{1}+z^{k_{0}}A_{0}\Biggr]^{-1}C$$
exists in an open neighborhood of the circle $|z|=r$, for any operator $W\in \{{\rm I}, A_{0},A_{1},A_{2}\}.$ If the sequences $(a_{1}(s))_{s\in {\mathbb N}_{0}}$ and $(a_{2}(s))_{s\in {\mathbb N}_{0}}$ are exponentially bounded, and the sequence $(f(k))_{k\in {\mathbb Z}}$ rapidly decays at infinity, then the requirements of Theorem \ref{movesvoltw} can be satisfied with $D'={\mathbb N}_{0}.$ Suppose, for example, that $|a_{1}(s)|+|a_{2}(s)| \leq Me^{\omega s},$ $s\in {\mathbb N}_{0}$ for some real numbers $M\geq 1$ and $\omega \geq 0$, $f(k)=e^{-k^{2}}x$, $k\in {\mathbb Z}$ for some $x\in X,$ and $r>e^{\omega}.$ Then, for $w=1,2$ and for every $k\in {\mathbb Z},$ we have:{\scriptsize 
\begin{align*} & 
\sum_{s=0}^{+\infty}\Bigl| a_{w}(s)\Bigr| \sum_{l\in {\mathbb Z}; k-s-l\in D'}  r^{k-l-s}  \\  &  \times \sup_{|\lambda|=r}  p\Biggl( W\Biggl[\sum_{w=1}^{s}\sum_{j_{w}=0}^{m_{w}}(-1)^{m_{w}-j_{w}}\binom{m_{w}}{j_{w}}\lambda^{k_{w}+j_{w}}F_{a_{w}}(z)A_{w}+\lambda^{k_{0}}A_{0}\Biggr]^{-1}Cf(l)\Biggr) 
\\& =\sum_{s=0}^{+\infty}\Bigl| a_{w}(s)\Bigr| \sum_{l=-\infty}^{k-s}  r^{k-l-s}  \\  &  \times \sup_{|\lambda|=r} 
\Biggl\| W\Biggl[\Bigl( z^{k_{2}}-2z^{k_{2}+1}+z^{k_{2}+2}\Bigr)F_{a_{2}}(z)A_{2}+\Bigl( z^{k_{1}+1}-z^{k_{1}}\Bigr)F_{a_{1}}(z)A_{1}+z^{k_{0}}A_{0}\Biggr]^{-1}C\Biggr\| \cdot \| f(l)\|
\\& 
\leq \sup_{|\lambda|=r} 
\Biggl\| W\Biggl[\Bigl( z^{k_{2}}-2z^{k_{2}+1}+z^{k_{2}+2}\Bigr)F_{a_{2}}(z)A_{2}+\Bigl( z^{k_{1}+1}-z^{k_{1}}\Bigr)F_{a_{1}}(z)A_{1}+z^{k_{0}}A_{0}\Biggr]^{-1}C\Biggr\| \cdot \| x\| \\& \times \sum_{s=0}^{+\infty}\Bigl| a_{w}(s)\Bigr| \sum_{l=-\infty}^{k-s}  r^{k-l-s}e^{-l^{2}} := M_{r}\sum_{s=0}^{+\infty}\Bigl| a_{w}(s)\Bigr| \sum_{l=-\infty}^{k-s}  r^{k-l-s}e^{-l^{2}}.
\end{align*}}After that, we compute
\begin{align*} &
\sum_{s=0}^{+\infty}\Bigl| a_{w}(s)\Bigr| \sum_{l=-\infty}^{k-s}  r^{k-l-s}e^{-l^{2}}
 \leq M\sum_{s=0}^{+\infty}\sum_{v=0}^{+\infty}e^{\omega s}r^{v}e^{-(v+s-k)^{2}}
\\& \leq M\sum_{s=0}^{+\infty}\sum_{v=0}^{+\infty}e^{\omega s}r^{v}e^{-v^{2}}e^{-s^{2}}e^{2vk}e^{2sk}
\leq M \sum_{s=0}^{+\infty}e^{\omega s}e^{-s^{2}}e^{2sk} \cdot \sum_{v=0}^{+\infty}r^{v}e^{-v^{2}}e^{2vk}
<+\infty .
\end{align*}
Therefore, conditions (i)-(iii) in formulation of Theorem \ref{movesvoltw} hold. We can apply the Cauchy theorem in order to see that condition (iv) in formulation of Theorem \ref{movesvoltw} also holds; cf. Example \ref{zan} for more details.
\item[(ii)] Suppose now that $0<\alpha_{1}<1,$ $1<\alpha_{2}<2,$ $a_{1}=c^{m_{1}-\alpha_{1}}=c^{1-\alpha_{1}}$ and $a_{2}=c^{m_{2}-\alpha_{2}}=c^{2-\alpha_{2}}.$ Then the equation \eqref{qad} becomes 
\begin{align}\label{qad1}
A_{2}\Bigl( \Delta_{W}^{\alpha_{2}}u\Bigr)\bigl(k+k_{2}\bigr)+A_{1}\Bigl( \Delta_{W}^{\alpha_{1}}u\Bigr)\bigl(k+k_{1}\bigr)+A_{0} u \bigl(k+k_{0}\bigr)=Cf(k),\quad k\in {\mathbb Z}.
\end{align}
Furthermore, we have
\begin{align*} &
\Bigl( z^{k_{2}}-2z^{k_{2}+1}+z^{k_{2}+2}\Bigr)F_{a_{2}}(z)A_{2}+\Bigl( z^{k_{1}+1}-z^{k_{1}}\Bigr)F_{a_{1}}(z)A_{1}+z^{k_{0}}A_{0}
\\& =\Bigl( z^{k_{2}}-2z^{k_{2}+1}+z^{k_{2}+2}\Bigr)\sum_{j=0}^{+\infty} \frac{\Gamma(j+2-\alpha_{2})}{\Gamma (2-\alpha_{2}) \cdot j!}z^{-j} A_{2}
\\& +\Bigl( z^{k_{1}+1}-z^{k_{1}}\Bigr)\sum_{j=0}^{+\infty} \frac{\Gamma(j+1-\alpha_{2})}{\Gamma (1-\alpha_{1}) \cdot j!}z^{-j} A_{1}+z^{k_{0}}A_{0}
\\& \sim \Bigl( z^{k_{2}}-2z^{k_{2}+1}+z^{k_{2}+2}\Bigr)A_{2}
+\Bigl( z^{k_{1}+1}-z^{k_{1}}\Bigr)z^{-j} A_{1}+z^{k_{0}}A_{0},
\end{align*}
as $|z|\rightarrow +\infty .$ If there exists a sufficiently large real number $r_{0}>1$ such that the operator
$$
\Bigl( z^{k_{2}}-2z^{k_{2}+1}+z^{k_{2}+2}\Bigr)A_{2}
+\Bigl( z^{k_{1}+1}-z^{k_{1}}\Bigr)z^{-j} A_{1}+z^{k_{0}}A_{0}
$$
is injective in an open neighborhood of a point $z_{0}\in {\mathbb C}$ with $|z_{0}|=r_{0},$ then
Theorem \ref{iniqvoltw} can be applied and there exists a unique strong solution of problem
\eqref{qad1} which belongs to the space $uniq_{\Omega}(X;\eqref{qad1}).$  This can occur, for example, if $A_{0}=A_{1}={\rm I},$ $k_{0}=k_{2}+1$ and there exists a complex number $z_{0}\in {\mathbb C}$ such that $|z_{0}|=r_{0} $ and the oprator $z_{0}+A_{2}$ is injective, as a simple calculation shows.
\end{itemize}
\end{example}

\section{Conclusions and final remarks}\label{rema}

In this paper, we have analyzed the multidimensional $Z$-transform of functions with values in sequentially complete locally convex spaces. We have provided many useful remarks, illustrative examples and applications to the abstract Volterra difference equations with multiple variables.

Let us finally emphasize that we can also consider some classes of the abstract Volterra difference equations with convolution product $\ast_{{\rm D}}^{l,j}$. For example, the multidimensional vector-valued $Z$-transform can serve one to provide the basic details about the existence and uniqueness of solutions to the following abstract Volterra difference equation with multiple variables:
\begin{align}\label{djednacina}
\sum_{l=0}^{n}\sum_{j=1}^{a_{l,n}}\sum_{s=1}^{b_{l,j}}A_{j,l,s}\Bigl( a_{j,l,s}\ast^{l,j}_{{\rm D}_{j,l,s}}u\Bigr)\bigl(k_{1},...,k_{n}\bigr)=Cf\bigl(k_{1},...,k_{n}\bigr),\ k=\bigl(k_{1},...,k_{n}\bigr) \in {\mathbb Z}^{n},
\end{align}
where $0\leq l\leq n,$ $C\in L(X)$ is injective, $a_{j,l,s} : D_{j,l,s}'\rightarrow {\mathbb C}$ with $\emptyset \neq D_{j,l,s}'\subseteq {\mathbb Z}^{l}$ for $1\leq l\leq n$, $A_{j,l,s}$ are linear operators on $X$ ($0\leq l\leq n;$ $1\leq j\leq a_{l,n}$; $1\leq s\leq b_{l,j}$), and $f : {\mathbb Z}^{n} \rightarrow X $. The continuous analogues of problem \eqref{djednacina} have recently been analyzed in \cite{mvlt}.

Performing the multidimenional vector-valued $Z$-transform to the both sides of \eqref{djednacina}, we obtain with the help of Theorem \ref{dsplj} that a solution $u :  {\mathbb Z}^{n} \rightarrow X$ of this problem should be looked in the form{\small
\begin{align*}
u(k)=\Biggl(Z^{-1}\Biggl[\Biggl( \sum_{l=0}^{n}\sum_{j=1}^{a_{l,n}}\sum_{s=1}^{b_{l,j}}F_{a_{j,l,s}}\bigl(z_{j_{1}},...,z_{j_{l}} \bigr) A_{j,l,s}\Biggr)^{-1}C\Biggr] \ast_{(D',{\mathbb Z}^{n})}f\Biggr)(k),\ k\in {\mathbb Z}^{n},
\end{align*}}where $Z^{-1}$ denotes the inverse multidimensional $Z$-transform. Further details can be left to interested readers.

\section*{Declarations}

\noindent {\bf Conflicts of Interest:} The author has no competing interests to declare that are relevant to the content of this article.\vspace{0.2cm}

\noindent {\bf Financial Interests:} 
This research is partially supported by grant no. 451-03-68/2020/14/200156 of Ministry
of Science and Technological Development, Republic of Serbia.\vspace{0.2cm}

\noindent {\bf Data Availability:} The data supporting the findings of this research study are available from
the author, upon reasonable request.

\end{document}